\documentclass{amsart}

\newtheorem{theorem}[equation]{Theorem}
\newtheorem{lemma}[equation]{Lemma}
\newtheorem{corollary}[equation]{Corollary}
\newtheorem{proposition}[equation]{Proposition}

\numberwithin{equation}{section}

\usepackage{amscd,amsmath,amssymb,verbatim}

\begin{document}

\title[Distinguished-root formulas]{Distinguished-root formulas for generalized \\ Calabi-Yau hypersurfaces}
\author{Alan Adolphson}
\address{Department of Mathematics\\
Oklahoma State University\\
Stillwater, Oklahoma 74078}
\email{adolphs@math.okstate.edu}
\author{Steven Sperber}
\address{School of Mathematics\\
University of Minnesota\\
Minneapolis, Minnesota 55455}
\email{sperber@math.umn.edu}
\date{\today}
\begin{abstract}
By a ``generalized Calabi-Yau hypersurface'' we mean a hypersurface in ${\mathbb P}^n$ of degree $d$ dividing $n+1$. The zeta function of a generic such hyper\-surface has a reciprocal root distinguished by minimal $p$-divisibility.  We study the $p$-adic variation of that distinguished root in a family and show that it equals the product of an appropriate power of $p$ times a product of special values of a certain $p$-adic analytic function ${\mathcal F}$.  That function ${\mathcal F}$ is the $p$-adic analytic continuation of the ratio $F(\Lambda)/F(\Lambda^p)$, where $F(\Lambda)$ is a solution of the $A$-hypergeometric system of differential equations corresponding to the Picard-Fuchs equation of the family.
\end{abstract}
\maketitle

\section{Introduction}

Dwork\cite{D1.1,D1.5} was the first to obtain $p$-adic analytic formulas for eigenvalues of Frobenius.  In \cite[Section~6]{D1.5}, Dwork developed an analytic theory of Frobenius for families of hypersurfaces: Frobenius acts semi-linearly on the space of local solutions of the Picard-Fuchs equation and preserves $p$-adic growth conditions.  In particular, $p$-adically bounded local solutions and $p$-adic unit eigenvalues of Frobenius are closely related.  In this article, we apply these ideas (with some modifications) to obtain $p$-adic analytic formulas for the unique eigenvalue of minimal $p$-divisibility for what we call generalized Calabi-Yau hypersurfaces.

The Legendre family of elliptic curves was the first case to be studied in detail.  In characteristic zero the Picard-Fuchs equation is of order 2, but Igusa\cite{Ig} noted that in odd characteristic $p$ it has only one series solution (up to $p$-th powers).   The truncation of the unique series solution ${}_2F_1(\frac{1}{2},\frac{1}{2};1;\Lambda)$ in characteristic zero at the $(p-1)$-st term makes sense in characteristic~$p$ and is the unique solution in characteristic~$p$.  Furthermore, for the elliptic curve in characteristic $p$, the number of rational points is determined modulo $p$ by this truncation.  Dwork used the Frobenius action on local solutions of Picard-Fuchs to give a much more precise result, namely, a formula for the unit root of the zeta function of a nonsupersingular elliptic curve of the Legendre family in terms of special values of the $p$-adic analytic continuation of the ratio ${}_2F_1(\frac{1}{2},\frac{1}{2};1;\Lambda)/{}_2F_1(\frac{1}{2},\frac{1}{2};1;\Lambda^p)$ (\cite[Eq.~(6.29)]{D1.5}).  Similar formulas have been found as well for the Dwork family of hypersurfaces (Dwork\cite{D1.5}, J-D Yu\cite{Y}), more general families of varieties (N. Katz\cite{K2}), and for families of toric exponential sums (Dwork\cite{D3}, \cite{AS01,AS02,AS3}).

Novel features of this work are that we obtain explicit fomulas for very general families of generalized Calabi-Yau hypersurfaces where the defining form is subject only to condition (1.9) below.  We avoid in particular any hypothesis of nonsingularity.  Dwork had suggested this might in fact be possible in his 1962 International Congress talk\cite[Section 5]{D1.1}.  This is achieved here in part by adopting the $A$-hypergeometric point of view, which makes it easy to write down the explicit solution (1.15) of the Picard-Fuchs equation satisfied by the differential form (1.10), and by avoiding any computations involving the cohomology of the hypersurfaces in the family.  

In addition, we apply here the dual theory associated with Dwork's $\theta_\infty$-splitting function.  While this is technically more complicated than the dual theory associated with the $\theta_1$-splitting function that Dwork used in \cite{D1.3}, the advantage is that our results are valid for all primes rather than just all sufficiently large primes.

We proceed now to make precise the main results.  Let
\begin{equation}
f_\lambda(x_0,\dots,x_n) = \sum_{j=1}^N \lambda_j x^{{\bf a}_j}\in{\mathbb F}_q^\times[x_0,\dots,x_n]
\end{equation}
be a homogeneous polynomial of degree $d\geq 2$ over the finite field ${\mathbb F}_q$, $q=p^a$, $p$ a prime.  Let ${\mathbb N}$ denote the set of nonnegative integers.  For each $j$ we write ${\bf a}_j = (a_{0j},\dots,a_{nj})\in{\mathbb N}^{n+1}$ with $\sum_{i=0}^n a_{ij} = d$ and $x^{{\bf a}_j} = x_0^{a_{0j}}\cdots x_n^{a_{nj}}$.  Let $X_\lambda\subseteq{\mathbb P}_{{\mathbb F}_q}^n$ be defined by the vanishing of $f_\lambda$ and let $X'_\lambda\subseteq{\mathbb A}_{{\mathbb F}_q}^{n+1}$ be the affine cone over $X_\lambda$.  By Ax\cite{A} we have for all $s$
\begin{equation}
{\rm card}\;X_\lambda'({\mathbb F}_{q^s}) \equiv 0\pmod{q^{\mu s}}, 
\end{equation}
where $\mu$ is the least nonnegative integer that is greater than or equal to $\frac{n+1}{d}-1$.  Equivalently,
\begin{equation}
{\rm card}\;X_\lambda({\mathbb F}_{q^s})\equiv \frac{1}{1-q^s} \pmod{q^{\mu s}}
\end{equation}
for all $s$.  

This latter congruence can be expressed in terms of the zeta function of~$X_\lambda$.  Define a function $P_\lambda(t)$ by
\[ P_\lambda(t) = \big( Z(X_\lambda/{\mathbb F}_q,t)(1-t)(1-qt)\cdots(1-q^{n-1}t) \big)^{(-1)^n}. \]
When the fiber $X_\lambda$ is smooth, $P_\lambda(t)$ is the characteristic polynomial of Frobenius acting on middle-dimensional primitive cohomology.  In this case, $P_\lambda(t)$ has degree $d^{-1}\big( (d-1)^{n+1} + (-1)^{n+1}(d-1)\big)$.  In the general setting, we have only that $P_\lambda(t)$ is a rational function (Dwork\cite{D1}).  
The congruence (1.3) is equivalent to the assertion that all reciprocal zeros $\rho$ and reciprocal poles $\sigma$ of $P_\lambda(t)$ satisfy
\begin{equation}
{\rm ord}_q\:\rho,{\rm ord}_q\:\sigma\geq\mu, 
\end{equation}
where ${\rm ord}_q$ is the $p$-adic valuation normalized by ${\rm ord}_q\:q = 1$ (Ax\cite{A}, Katz\cite[Proposition~2.4]{K}).  

The integer $\mu$ has Hodge-theoretic significance.  Let $Y\subseteq{\mathbb P}^n_{{\mathbb C}}$ be a smooth hypersurface of degree $d$ and let $\{h^{i,n-1-i}\}_{i=0}^{n-1}$ be the Hodge numbers of the primitive part of middle-dimensional cohomology of~$Y$ (the~$h^{i,n-1-i}$ depend only on $n$ and~$d$).  Then $i=\mu$ is the smallest value of $i$ for which $h^{i,n-1-i}\neq 0$ and, as such, is referred to as the Hodge type of $Y$.  Furthermore, for $X_\lambda$ smooth over ${\mathbb F}_q$ the rational function $P_\lambda(t)$ is a polynomial and, by Illusie\cite{I}, the generic smooth $X_\lambda$ has exactly $h^{\mu,n-1-\mu}$ reciprocal zeros $\rho$ satisfying ${\rm ord}_q\:\rho = \mu$.  

In this paper we focus our attention on cases where $h^{\mu,n-1-\mu}=1$, i.e., where the polynomial $P_\lambda(t)$ has a unique reciprocal zero $\rho$ with smallest $q$-ordinal $\mu$ for generic smooth $X_\lambda$.  By standard formulas for Hodge numbers, this occurs when $d$ is a divisor of $n+1$.  From the definition of $\mu$, we then have 
\begin{equation}
n+1 = d(\mu+1),
\end{equation}
which we assume from now on.  We refer to these varieties as generalized Calabi-Yau hypersurfaces.  (The case $\mu=0$ is the classical case of projective Calabi-Yau hypersurfaces.)  Assuming only this condition, one can refine the description of~$P_\lambda(t)$.

For $j=1,\dots,N$, put 
\[ {\bf a}_j^+ = ({\bf a}_j,1) = (a_{0j},a_{1j},\dots,a_{nj},1)\in{\mathbb N}^{n+2}. \]
Let $\Lambda_1,\dots,\Lambda_N$ be indeterminates and set
\begin{equation}
H(\Lambda) = \sum_{\substack{u=(u_1,\dots,u_N)\in{\mathbb N}^N\\ \sum_{j=1}^N u_j{\bf a}_j^+ = (p-1)(1,\dots,1,\mu+1)}} \frac{\Lambda_1^{u_1}\cdots\Lambda_N^{u_N}}{u_1!\cdots u_N!}\in\big({\mathbb Q}\cap{\mathbb Z}_p\big)[\Lambda_1,\dots,\Lambda_N].
\end{equation}
Note that the conditions on the summation imply $0\leq u_j\leq p-1$ for $j=1,\dots,N$.  We denote by $\bar{H}(\Lambda)\in{\mathbb F}_p[\Lambda_1,\dots,\Lambda_N]$ the reduction mod $p$ of $H(\Lambda)$.

We express the rational function $P_\lambda(t)$ as a ratio $P_\lambda(t)=P^{(1)}_\lambda(t)/P^{(2)}_\lambda(t)$, where $P^{(1)}_\lambda(t)$ and $P^{(2)}_\lambda(t)$ are  relatively prime polynomials with integer coefficients and constant term $1$.  By (1.4) we have
\[ P_\lambda^{(1)}(q^{-\mu}t),P_\lambda^{(2)}(q^{-\mu}t)\in 1+t{\mathbb Z}[t]. \]
We prove the following result in Section 7.
\begin{proposition}
Let $f_\lambda$ be as in $(1.1)$ and suppose $(1.5)$ holds.  Let $\hat{\lambda}\in{\mathbb Q}_p(\zeta_{q-1})^N$ be the Teichm\"uller lifting of $\lambda$.   Then $P_\lambda^{(2)}(q^{-\mu}t) \equiv 1\pmod{q}$ and
\[ P_\lambda^{(1)}(q^{-\mu}t)\equiv 1 -  t \prod_{i=0}^{a-1} \big((-1)^{\mu+1}H(\hat{\lambda}^{p^i})\big) \pmod{p}. \]
\end{proposition}

As an immediate consequence of Proposition 1.7, we get a criterion for the zeta function of a generalized Calabi-Yau hypersurface to have a reciprocal root distinguished by minimal $p$-divisibility.
\begin{proposition}
Under the hypotheses of Proposition $1.7$, the rational function $P_\lambda(t)$ has a unique reciprocal root of $q$-ordinal $\mu$ if and only if $\bar{H}(\lambda)\neq 0$.  Furthermore, when $\bar{H}(\lambda)\neq 0$, that reciprocal root is a reciprocal zero, not a reciprocal pole, of $P_\lambda(t)$.
\end{proposition}

When $\bar{H}(\lambda)\neq 0$, we denote by $\rho_{\min}(\lambda)$ the unique reciprocal root of $P_\lambda(t)$ having $q$-ordinal $\mu$.  Let $\bar{\mathbb F}_q$ denote an algebraic closure of ${\mathbb F}_q$.  We call the set
\[ \{ \lambda\in\bar{\mathbb F}_q^N \mid \bar{H}(\lambda)\neq 0\} \]
the {\it Hasse domain\/} for the family.

It can happen that the sum defining $H(\Lambda)$ is empty, for example, if $f_\lambda$ is the diagonal hypersurface of degree $d$ dividing $n+1$ and $p\not\equiv 1\pmod{d}$.  To guarantee that for all primes $p$ the polynomial $H(\Lambda)$ is not identically zero, we make the assumption that $\mu+1$ of the vectors $\{{\bf a}_j\}_{j=1}^N$ sum to the vector $(1,\dots,1)$, say,
\begin{equation}
\sum_{j=1}^{\mu+1} {\bf a}_j = (1,\dots,1).
\end{equation}
The monomial $\prod_{j=1}^{\mu+1} (\Lambda_j^{p-1}/(p-1)!)$ then appears in $H(\Lambda)$ and, as a consequence, the subset of $({\mathbb F}_q^\times)^N$ where $\bar{H}(\lambda)\neq 0$ is nonempty.  
Equation (1.9) is equivalent to the condition that $x^{{\bf a}_1}\cdots x^{{\bf a}_{\mu+1}} = x_0x_1\cdots x_n$.  
For example, in the case of Calabi-Yau hypersurfaces where $d=n+1$ and $\mu=0$, this just says that $x_0x_1\cdots x_n$ must be one of the monomials that appear in $f_\lambda$.  Our main goal in this paper is to give a $p$-adic analytic description of $\rho_{\min}(\lambda)$ in terms of $A$-hypergeometric functions when $\bar{H}(\lambda)\neq 0$.

Let $U\subseteq{\mathbb P}_{\mathbb C}^n$ be the open complement of a smooth hypersurface $Y$ defined by a homogeneous polynomial $g$ of degree $d$.  Under the hypothesis~(1.5), there is an $n$-form on $U$ which can be expressed in homogeneous coordinates as
\begin{equation} \frac{\sum_{i=0}^n (-1)^ix_i\,dx_0\cdots\widehat{dx}_i\cdots dx_n}{g^{\mu+1}}. 
\end{equation}
This $n$-form determines a cohomology class in $H^n_{\rm DR}(U)$, and also, by applying the residue map, a cohomology class in $H^{n-1}_{\rm DR}(Y)$.  The one-dimensional space spanned by this cohomology class is the Hodge subspace of ``co-level'' $\mu$.  When $Y$ varies in a family, this cohomology class satisfies a Picard-Fuchs equation.  The $A$-hypergeometric equation that describes the variation of $\rho_{\min}(\lambda)$ when $\bar{H}(\lambda)\neq 0$ is the $A$-hypergeometric version of this Picard-Fuchs equation.

We describe the relevant $A$-hypergeometric system.  Let $A = \{{\bf a}_j^+\}_{j=1}^N$ and 
let $L\subseteq{\mathbb Z}^{N}$ be the lattice of relations on the set $A$:
\[ L= \bigg\{l=(l_1,\dots,l_N)\in{\mathbb Z}^N\mid \sum_{j=1}^N l_j{\bf a}_j^+ = {\bf 0}\bigg\}. \]
For each $l=(l_1,\dots,l_N)\in L$, we define a partial differential operator $\Box_l$ in variables $\{\Lambda_j\}_{j=1}^N$ by
\begin{equation}
\Box_l = \prod_{l_j>0} \bigg(\frac{\partial}{\partial\Lambda_j}\bigg)^{l_j} - \prod_{l_j<0} \bigg(\frac{\partial}{\partial\Lambda_j}\bigg)^{-l_j}.
\end{equation}
For $\beta = (\beta_0,\beta_1,\dots,\beta_{n+1})\in{\mathbb C}^{n+2}$, the corresponding Euler (or homogeneity) operators are defined by
\begin{equation}
 Z_i = \sum_{j=1}^N a_{ij}\Lambda_j\frac{\partial}{\partial\Lambda_j} - \beta_i 
\end{equation}
for $i=0,\dots,n+1$.  The {\it $A$-hypergeometric system with parameter $\beta$\/} consists of Equations (1.11) for $l\in L$ and (1.12) for $i=0,1,\dots,n+1$.

The $A$-hypergeometric system satisfied by the $n$-form (1.10) is obtained by taking the parameter $\beta$ to be (using (1.9) above)
\begin{equation}
{\bf b} := -\sum_{j=1}^{\mu+1} {\bf a}_j^+ = (-1,\dots,-1,-\mu-1)\in{\mathbb C}^{n+2}.
\end{equation}
Let $v=(-1,\dots,-1,0,\dots,0)\in{\mathbb C}^N$ ($-1$ repeated $\mu+1$ times followed by $0$ repeated $N-\mu-1$ times).  Then 
\begin{equation}
\sum_{j=1}^N v_j{\bf a}_j^+ = {\bf b}
\end{equation}
and $v$ has minimal negative support in the terminology of Saito-Sturm\-fels-Taka\-yama\cite{SST}, so by \cite[Proposition~3.4.13]{SST} we get a series solution of this $A$-hyper\-geometric system.  Let $L'$ be the subset of $L$ consisting of all $l=(l_1,\dots,l_N)$ such that $l_j\leq 0$ for $j=1,\dots,\mu+1$ and $l_j\geq 0$ for $j=\mu+2,\dots,N$.  The series solution is
$(\Lambda_1\cdots\Lambda_{\mu+1})^{-1}F(\Lambda)$, where
\begin{equation}
F(\Lambda) = \sum_{l\in L'} \frac{(-1)^{\sum_{j=1}^{\mu+1} l_j} \prod_{j=1}^{\mu+1} (-l_j)!}{\prod_{j=\mu+2}^N l_j!} \prod_{j=1}^N \Lambda_j^{l_j}.
\end{equation}

Since the last coordinate of each ${{\bf a}}_j^+$ equals 1, the condition $l\in L$ implies that $\sum_{j=1}^N l_j = 0$, hence $F(\Lambda)$ is homogeneous of degree 0 in the $\Lambda_j$.  For $j=1,\dots,\mu+1$, the $\Lambda_j$ occur to nonpositive powers in $F(\Lambda)$ and for $j=\mu+2,\dots,N$, the $\Lambda_j$ occur to nonnegative powers in $F(\Lambda)$.  The coefficients of the series $F(\Lambda)$ are integers by \cite[Proposition~5.2]{AS}, therefore it converges and assumes unit valules on the set 
\begin{multline*}
 {\mathcal D} = \{(\Lambda_1,\dots,\Lambda_N)\in {\mathbb C}_p^N\mid \text{$|\Lambda_j|>1$ for $1\leq j\leq \mu+1$} \\ \text{ and $|\Lambda_j|<1$ for $\mu+2\leq j\leq N$}\} 
\end{multline*}
(where ${\mathbb C}_p$ denotes the completion of an algebraic closure of ${\mathbb Q}_p$).  
Let
\begin{multline*}
{\mathcal D}_+ =  \{\Lambda\in {\mathbb C}_p^N\mid \text{$|\Lambda_j|\geq 1$ for $1\leq j\leq \mu+1$}, \\
\text{$|\Lambda_j|\leq 1$ for $\mu+2\leq j\leq N$, and $|(\Lambda_1\cdots\Lambda_{\mu+1})^{-(p-1)}H(\Lambda)| = 1$}\}. 
\end{multline*}
Note that the Laurent polynomial $(\Lambda_1\cdots\Lambda_{\mu+1})^{-(p-1)}H(\Lambda)$ has only nonpositive powers of $\Lambda_j$ for $j=1,\dots,\mu+1$, only nonnegative powers of $\Lambda_j$ for $j=\mu+2,\dots,N$, and constant term~$\big((p-1)!\big)^{-(\mu+1)}$.  This implies that $(\Lambda_1\cdots\Lambda_{\mu+1})^{-(p-1)}H(\Lambda)$ assumes unit values on ${\mathcal D}$, hence ${\mathcal D}\subseteq{\mathcal D}_+$.  

Our main result is the following.
\begin{theorem}
Under hypotheses $(1.5)$ and $(1.9)$, the ratio $F(\Lambda)/F(\Lambda^p)$ extends to an analytic function ${\mathcal F}(\Lambda)$ on~${\mathcal D}_+$.  Let $\lambda\in({\mathbb F}_q^\times)^N$ and let $\hat{\lambda}\in{\mathbb Q}_p(\zeta_{q-1})^N$ be its Teichm\"uller lifting.  If $\bar{H}(\lambda)\neq 0$, then $\hat{\lambda}^{p^i}\in{\mathcal D}_+$ for $i=0,\dots,a-1$ and
\[ \rho_{\min}(\lambda) = q^\mu\prod_{i=0}^{a-1} {\mathcal F}(\hat{\lambda}^{p^i}). \]
\end{theorem}

{\it Example $1$.}  When $d=n+1$ and $\mu=0$, Theorem 1.16 gives a unit root formula assuming only that $x_0\cdots x_n$ is one of the monomials appearing in $f_\lambda$.  If $f_\lambda$ defines a smooth hypersurface, then $P_\lambda(t)$ is a polynomial and this is its unique unit root.   Consider for instance the Dwork family of hypersurfaces:
\[ f_\lambda(x_0,\dots,x_n) = \lambda_1 x_0\cdots x_n + \lambda_2x_0^{n+1} + \lambda_3x_1^{n+1} + \cdots + \lambda_{n+2} x_n^{n+1}. \]
One computes that $L' = \{(-(n+1)l,l,\dots,l)\in {\mathbb Z}^{n+2}\mid l\in{\mathbb N}\}$ and
\[ F(\Lambda) = \sum_{l=0}^\infty \frac{(-1)^{(n+1)l}((n+1)l)!}{(l!)^{n+1}}\bigg(\frac{\Lambda_2\cdots\Lambda_{n+2}}{\Lambda_1^{n+1}}\bigg)^l. \]
By Theorem 1.16, the ratio ${\mathcal F}(\Lambda) = F(\Lambda)/F(\Lambda^p)$ extends to ${\mathcal D}_+$ and the product $\prod_{i=0}^{a-1} {\mathcal F}(\hat{\lambda}^{p^i})$
gives the unit reciprocal zero of $P_\lambda(t)$ when $\bar{H}(\lambda)\neq 0$.  

The more usual way of normalizing the Dwork family is
\[ x_0^{n+1} + \cdots + x_n^{n+1} - (n+1)\Lambda^{-1/(n+1)}x_0\cdots x_n, \]
which we can recover from the specialization $\Lambda_1\mapsto -(n+1)\Lambda^{-1/(n+1)}$ and $\Lambda_j\mapsto 1$ for $j=2,\dots,n+2$, giving
\begin{align*} 
F(-(n+1)\Lambda^{-1/(n+1)},1,\dots,1) &= \sum_{l=0}^{\infty} \frac{\big((n+1)l\big)!}{(l!)^{n+1}(n+1)^{(n+1)l}}\Lambda^l \\
 &= {}_nF_{n-1}(1/(n+1),\dots,n/(n+1);1,\dots,1;\Lambda)
\end{align*}
The assertion of Theorem 1.16 for this normalization of the Dwork family was recently proved by J.-D. Yu\cite{Y}.

{\it Example $2$.}  Let
\[ f_\lambda(x_0,\dots,x_5) = \lambda_1x_0x_1x_2 + \lambda_2x_3x_4x_5+\sum_{i=0}^5 \lambda_{i+3}x_i^3. \]
One computes that
\[ L' = \{ l_1(-3,0,1,1,1,0,0,0) + l_2(0,-3,0,0,0,1,1,1)\mid l_1,l_2\in{\mathbb N} \}, \]
hence
\[ F(\Lambda) = \sum_{l_1,l_2=0}^\infty \frac{(-1)^{l_1+l_2}(3l_1)!(3l_2)!}{(l_1!)^3(l_2!)^3} \frac{(\Lambda_3\Lambda_4\Lambda_5)^{l_1}(\Lambda_6\Lambda_7\Lambda_8)^{l_2}}{\Lambda_1^{3l_1}\Lambda_2^{3l_2}}. \]
By Theorem 1.16, the ratio ${\mathcal F}(\Lambda) = F(\Lambda)/F(\Lambda^p)$ extends to ${\mathcal D}_+$ and $q\prod_{i=0}^{a-1} {\mathcal F}(\hat{\lambda}^{p^i})$ equals the reciprocal zero $\rho_{\min}(\lambda)$ of $P_\lambda(t)$ with ${\rm ord}_q\:\rho_{\min}(\lambda) = 1$ when $\bar{H}(\lambda)\neq 0$.

{\it Remark.}  Even when there is no choice of $\mu+1$ elements of the set $\{{\bf a}_j\}_{j=1}^N$ satisfying (1.9), results similar to Theorem 1.16 may be true.  For example, suppose that $p\equiv 1\pmod{d}$ and that
\[ {\bf a}_j = (0,\dots,0,d,0,\dots,0)\quad\text{for $j=1,\dots,n+1$}, \]
where the `$d$' occurs in the $(j-1)$-st coordinate (i.~e., the polynomial $f_\lambda$ is a deformation of the diagonal hypersurface).  Equation (1.14) will remain valid if we choose 
\[ v=(-1/d,\dots,-1/d,0,\dots,0), \]
where the `$-1/d$' is repeated $n+1$ times.  Since this vector $v$ has minimal negative support, there is a corresponding series solution of the $A$-hypergeometric system with parameter ${\bf b}$ given by \cite[Proposition~3.4.13]{SST}.  And by \cite[Corollary~3.6]{AS} this series will have $p$-integral coefficients for $p\equiv 1\pmod{d}$.  Arguments similar to those of this article will  then show that an analogue of Theorem~1.16 is true for this series solution when $p\equiv 1\pmod{d}$.   

This paper is organized as follows.  In Section 2 we collect some notation that will be used throughout the paper.  In Section 3 we recall some estimates from Dwork\cite{D1.25} that will play a key role in what follows.  In Section 4 we show that Theorem 1.16 is equivalent to the same statement with $F(\Lambda)$ replaced by a related series $G(\Lambda)$.  The series $G(\Lambda)$ depends on the prime $p$ but satisfies better $p$-adic estimates than~$F(\Lambda)$.  (Without introducing $G(\Lambda)$, we would only be able to prove Theorem~1.16 for almost all primes.)  We use these estimates in Sections 5 and 6 to prove the analytic continuation of~$G(\Lambda)/G(\Lambda^p)$ and some related series.  Finally, in Section 7, we prove Proposition~1.7 and derive the formula for $\rho_{\min}(\lambda)$ in terms of special values of $G(\Lambda)/G(\Lambda^p)$ at Teichm\"uller points.

In a future work, we hope to treat as well the case in which the first nonvanishing Hodge number $h:=h^{\mu,n-1-\mu}$ is $>1$.  In this case, the (higher) Hasse-Witt matrix is $h\times h$ and, as in the case $h=1$, its entries may be described in terms of power series solutions of appropriate $A$-hypergeometric systems. 

\section{Notation}

For the convenience of the reader we collect in this section some notation that will be used throughout the paper.

Let ${\mathbb N}A\subseteq{\mathbb Z}^{n+2}$ be the semigroup generated by $A$ and let $M\subseteq{\mathbb Z}^{n+2}$ be the abelian group generated by $A$.  Note that $M$ lies in the hyperplane $\sum_{i=0}^n u_i = du_{n+1}$ in ${\mathbb R}^{n+2}$.  Set $M_- = M\cap({\mathbb Z}_{<0})^{n+2}$, $M_+ = M\cap{\mathbb N}^{n+2}$.    We denote by $\delta_-$ the truncation operator on formal Laurent series in variables $x_0,\dots,x_{n+1}$ that preserves only those terms having all exponents negative:
\[ \delta_-\bigg(\sum_{k\in{\mathbb Z}^{n+2}}c_kx^k\bigg) = \sum_{k\in({\mathbb Z}_{<0})^{n+2}} c_kx^k. \]
We shall use the same notation for formal Laurent series in a single variable $t$:
\[ \delta_-\bigg(\sum_{k=-\infty}^\infty c_kt^k\bigg) = \sum_{k=-\infty}^{-1} c_kt^k. \]

Let $E\subseteq{\mathbb Z}^N$ be the set
\[ E=\{(l_1,\dots,l_N)\mid\text{$l_j\leq 0$ for $1\leq j\leq\mu+1$ and $l_j\geq 0$ for $\mu+2\leq j\leq N$}\}. \]
Note that, in the notation of the Introduction, $L' = L\cap E$.  
We shall need to consider series in the $\Lambda_j$ that, like $F(\Lambda)$ (see (1.15)), have exponents lying in $E$.

Let ${\mathbb C}_p$ be the completion of an algebraic closure of ${\mathbb Q}_p$. For each $u\in M$, put
\[ R_u = \bigg\{ \xi(\Lambda) = \sum_{\substack{\nu\in E\\ \sum_{j=1}^N \nu_j {\bf a}^+_j = u}} c_\nu\prod_{j=1}^N \Lambda_j^{\nu_j}\mid \text{$c_\nu\in{\mathbb C}_p$ and $\{\lvert c_\nu\rvert\}_\nu$ is bounded}\bigg\} \]
We define the {\it degree\/} of a monomial $\Lambda^\nu$ to be $\sum_{j=1}^N \nu_j{\bf a}^+_j\in M$.  The series in~$R_u$ are analytic functions convergent and bounded on ${\mathcal D}$ and are homogeneous of degree $u$.

Let $R_u'$ be the set of functions on ${\mathcal D}_+$ that are uniform limits of sequences of 
rational functions $h(\Lambda)$ in the $\Lambda_j$, $1\leq j\leq N$, that are defined on ${\mathcal D}_+$ and homogeneous of degree~$u$, i.e., that satisfy
\[ h(\dots,t_0^{a_{0j}}\cdots t_n^{a_{nj}}t_{n+1}\Lambda_j,\dots) = t_0^{u_1}\cdots t_n^{u_n}t_{n+1}^{u_{n+1}} h(\Lambda). \]
Equivalently, this says that the restriction of $h(\Lambda)$ to ${\mathcal D}$ lies in $R_u$.  In particular, $R_u'$ is a subset of $R_u$.  The set $R_0$ is a ring, $R_u$ is a module over $R_0$, $R_0'$ is a subring 
of~$R_0$ and $R'_u$ is a module over $R'_0$.  We define a norm on $R_u$ by setting
\[ \lvert \xi\rvert = \sup_{\Lambda\in{\mathcal D}} \lvert\xi(\Lambda)\rvert. \]
Note that for $\xi\in R'_u$ one has $\sup_{\Lambda\in{\mathcal D}}\lvert\xi(\Lambda)\rvert = 
\sup_{\Lambda\in{\mathcal D}_+} \lvert\xi(\Lambda)\rvert$.  Both $R_u$ and $R_u'$ are complete in this norm.  Note that an element $\xi(\Lambda) = \sum_{\nu\in L'} c_\nu\Lambda^\nu\in R_0$ is invertible if and only if $\lvert\xi(\Lambda)\rvert = \lvert c_0\rvert$.  

From the discussion in the Introduction we see that $F(\Lambda)/F(\Lambda^p)\in R_0$.  To prove the first assertion of Theorem 1.16 we need to show that $F(\Lambda)/F(\Lambda^p)\in R'_0$.  In Section~4, we show that this is equivalent to the same assertion for a related function $G(\Lambda)$, for which the desired assertion is proved in Corollary~5.17.

Let $\gamma_0$ be a zero of the series $\sum_{i=0}^\infty t^{p^i}/p^i$ having ${\rm ord}_p\:\gamma_0 = 1/(p-1)$, where ${\rm ord}_p$ is the $p$-adic valuation normalized by ${\rm ord}_p\:p = 1$ (the role of $\gamma_0$ will be discussed more fully in the next section).  Define $S$ to be the ${\mathbb C}_p$-vector space of formal series
\[ S = \bigg\{\xi(\Lambda,x) = \sum_{u\in M_-} \xi_{u}(\Lambda) \gamma_0^{u_{n+1}}x^{u} \mid 
\text{$\xi_{u}(\Lambda)\in R_{u}$ and $\{\lvert\xi_{u}\rvert\}_u$ is bounded}\bigg\}. \]
Let $S'$ be defined analogously with the condition ``$\xi_{u}(\Lambda)\in R_{u}$'' being replaced by 
``$\xi_{u}(\Lambda)\in R_{u}'$''.  Define a norm on $S$ by setting
\[ \lvert\xi(\Lambda,x)\rvert = \sup_u\{\lvert\xi_{u}\rvert\}. \]
Both $S$ and $S'$ are complete under this norm.

\section{$p$-adic estimates}

We begin by recording some basic $p$-adic estimates from \cite[Section 4]{D1.25} that will play a role in what follows.  Let $ {\rm AH}(t)= \exp(\sum_{i=0}^{\infty}t^{p^i}/p^i)$ be the Artin-Hasse series, a power series in $t$ that has $p$-integral coefficients, and set 
\[ \theta(t) = {\rm AH}({\gamma}_0t)=\sum_{i=0}^{\infty}\theta_i t^i. \]
We then have
\begin{equation}
{\rm ord}_p\: \theta_i\geq \frac{i}{p-1}.
\end{equation}

We define $\hat{\theta}(t) = \prod_{j=0}^\infty \theta(t^{p^j})$, which gives $\theta(t) = \hat{\theta}(t)/\hat{\theta}(t^p)$.  If we set 
\begin{equation}
\gamma_j = \sum_{i=0}^j \frac{\gamma_0^{p^i}}{p^i},
\end{equation}
then
\begin{equation}
\hat{\theta}(t) = \exp\bigg(\sum_{j=0}^{\infty} \gamma_j t^{p^j}\bigg) = \prod_{j=0}^\infty \exp(\gamma_j t^{p^j}).
\end{equation}
Since $\big(p^i/(p-1)\big) -i$ is an increasing function of $i$ for $i\geq 1$, we have from the definition of $\gamma_0$ that
\begin{equation}
{\rm ord}_p\: \gamma_j = \frac{p^{j+1}}{p-1} - (j+1).
\end{equation}

We estimate each of the series $\exp(\gamma_jt^{p^j}) = \sum_{k=0}^\infty (\gamma_j t^{p^j})^k/k!$.
We have (where~$s_k$ denotes the sum of the digits in the $p$-adic expansion of $k$)
\begin{align}
{\rm ord}_p\: \frac{\gamma_j^k}{k!}  &=   k\bigg( \frac{p^{j+1}}{p-1} -(j+1)\bigg) - \frac{k-s_k}{p-1} \\ \nonumber
 &= k(p^j+p^{j-1}+\cdots+p-j) + \frac{s_k}{p-1}.
\end{align}
It follows that if we write $\exp(\gamma_j t^{p^j}) = \sum_{i=0}^\infty a^{(j)}_it^i$, then $a^{(j)}_i=0$ if $p^j\nmid i$, while if $i=p^jk$ then we have
\begin{align}
{\rm ord}_p\: a^{(j)}_i &= \frac{i}{p^j}(p^j+p^{j-1} + \cdots + p-j) + \frac{s_i}{p-1} \\ \nonumber
 &= i\bigg(1+\frac{1}{p} + \cdots +\frac{1}{p^{j-1}} -\frac{j}{p^j}\bigg) +\frac{s_i}{p-1}
\end{align}
(using $s_i=s_k$).  Equation (3.6) implies that ${\rm ord}_p\:a_i^{(j_1)}\geq {\rm ord}_p\:a_i^{(j_2)}$ if $j_1\geq j_2$.  
It follows that for all $j\geq 1$, 
\begin{equation}
{\rm ord}_p\:a_i^{(j)}\geq{\rm ord}_p\:a_i^{(1)} \geq \frac{i(p-1)}{p} + \frac{s_i}{p-1}\geq \frac{s_i}{p-1} = {\rm ord}_p\:a^{(0)}_i.
\end{equation}
If we write $\hat{\theta}(t) = \sum_{i=0}^\infty \hat{\theta}_i(\gamma_0 t)^i/i!$, then (3.3) and (3.7) imply
\begin{equation}
{\rm ord}_p\:\hat{\theta}_i\geq 0.
\end{equation}

We shall also need the series
\begin{equation}
\hat{\theta}_1(t) = \prod_{j=1}^\infty \exp(\gamma_jt^{p^j}) = :\sum_{i=0}^\infty \frac{\hat{\theta}_{1,i}}{i!}(\gamma_0 t)^i. 
\end{equation}
Note that $\hat{\theta}(t) = \exp(\gamma_0t)\hat{\theta}_1(t)$.  
Using the relation $s_{i_1}+s_{i_2}\geq s_{i_1+i_2}$, Equation (3.7) implies that 
\begin{equation}
{\rm ord}_p\:\hat{\theta}_{1,i}\geq \frac{i(p-1)}{p}. 
\end{equation}

Define a series $\hat{\theta}_1(\Lambda,x)$ by the formula
\begin{equation}
\hat{\theta}_1(\Lambda,x) = \prod_{j=1}^N \hat{\theta}_1(\Lambda_jx^{{\bf a}^+_j}).
\end{equation}
Expanding the product~(3.11) according to powers of $x$ we get
\begin{equation}
\hat{\theta}_1(\Lambda,x) = \sum_{u=(u_0,\dots,u_{n+1})\in{\mathbb N}A} \hat{\theta}_{1,u}(\Lambda)\gamma_0^{u_{n+1}} x^u,
\end{equation}
where
\begin{equation}
\hat{\theta}_{1,u}(\Lambda) = \sum_{\substack{k_1,\dots,k_N\in{\mathbb N}\\ \sum_{j=1}^N k_j{\bf a}^+_j = u}} \bigg(\prod_{j=1}^N \frac{\hat{\theta}_{1,k_j}}{k_j!}\bigg)\Lambda_1^{k_1}\cdots\Lambda_N^{k_N}.
\end{equation}

We have similar results for the reciprocal power series
\[ \hat{\theta}_1(t)^{-1} = \prod_{j=1}^\infty \exp(-\gamma_jt^{p^j}). \]
If we write
\begin{equation}
\hat{\theta}_1(t)^{-1} = \sum_{i=0}^\infty \frac{\hat{\theta}'_{1,i}}{i!}(\gamma_0 t)^i,
\end{equation}
then the coefficients satisfy
\begin{equation}
{\rm ord}_p\: \hat{\theta}'_{1,i}\geq \frac{i(p-1)}{p}.
\end{equation}
We also have
\begin{equation}
\hat{\theta}_1(\Lambda,x)^{-1} = \prod_{j=1}^N \hat{\theta}_1(\Lambda_jx^{{\bf a}^+_j})^{-1},
\end{equation}
which we again expand in powers of $x$ as
\begin{equation}
\hat{\theta}_1(\Lambda,x)^{-1} = \sum_{u=(u_0,\dots,u_{n+1})\in{\mathbb N}A} \hat{\theta}'_{1,u}(\Lambda)\gamma_0^{u_{n+1}} x^u
\end{equation}
with
\begin{equation}
\hat{\theta}'_{1,u}(\Lambda) = \sum_{\substack{k_1,\dots,k_N\in{\mathbb N}\\ \sum_{j=1}^N k_j{\bf a}_j^+ = u}} \bigg(\prod_{j=1}^N \frac{\hat{\theta}'_{1,k_j}}{k_j!}\bigg) \Lambda_1^{k_1}\cdots\Lambda_N^{k_N}.
\end{equation}

We also define
\begin{equation}
\theta(\Lambda,x) = \prod_{j=1}^N \theta(\Lambda_jx^{{\bf a}^+_j}).
\end{equation}
Expanding the right-hand side in powers of $x$, we have 
\begin{equation}
\theta(\Lambda,x) = \sum_{u\in{\mathbb N}A} \theta_u(\Lambda)x^u,
\end{equation}
where
\begin{equation}
\theta_u(\Lambda) = \sum_{\nu\in{\mathbb N}^N} \theta^{(u)}_\nu\Lambda^\nu
\end{equation}
and
\begin{equation}
\theta^{(u)}_\nu = \begin{cases} \prod_{j=1}^N \theta_{\nu_j} & \text{if $\sum_{j=1}^N \nu_j{\bf a}^+_j = u$,} \\
0 & \text{if $\sum_{j=1}^N \nu_j{\bf a}^+_j \neq u$,} \end{cases}
\end{equation}
so $\theta_u(\Lambda)$ is homogeneous of degree $u$.
The equation $\sum_{j=1}^N \nu_j{\bf a}^+_j = u$ has only finitely many solutions $\nu\in{\mathbb N}^N$, so $\theta_u(\Lambda)$ is a polynomial in the $\Lambda_j$.  Equations (3.1) and (3.22) show that
\begin{equation}
{\rm ord}_p\:\theta^{(u)}_\nu \geq\frac{\sum_{j=1}^N \nu_j}{p-1} = \frac{u_{n+1}}{p-1}.
\end{equation}

We observe one congruence that will allow us to simplify some later formulas.  From (3.2) and (3.4) with $j=1$ we have
\[ \gamma_0 + \frac{\gamma_0^p}{p}\equiv 0\pmod{\gamma_0p^{p-1}}. \]
Multiplying this congruence by $p/\gamma_0$ gives $\gamma_0^{p-1} \equiv -p \pmod{p^p}$, so, {\it a fortiori},
\begin{equation}
\gamma_0^{p-1}\equiv -p \pmod{p^2}\quad\text{for all primes $p$}.
\end{equation}

\section{Generating series for $A$-hypergeometric functions}

In Dwork's theory hypergeometric functions often appear in contiguous families, as coefficients of a generating series.  We describe the relevant generating series that will appear in our situation.  

Consider the formal series $\zeta(t)$ defined by
\begin{equation}
\zeta(t) = \sum_{l=0}^\infty (-1)^ll!t^{-l-1}. 
\end{equation}
We note that the series $\zeta(t)$ shares a property with the exponential series $\exp t$: differentiating a term of the series with respect to $t$ equals the term of the series involving the next lower power of $t$.

We define the formal generating series $F(\Lambda,x)$ by the formula
\begin{equation}
F(\Lambda,x) = \delta_-\bigg(\prod_{j=1}^{\mu+1} \zeta(\gamma_0\Lambda_jx^{{\bf a}_j^+}) \prod_{j=\mu+2}^N \exp(\gamma_0\Lambda_jx^{{\bf a}_j^+})\bigg),
\end{equation}
where $\delta_-$ is as defined in Section 2.  A straightforward calculation shows that
\begin{equation}
F(\Lambda,x) = \sum_{u\in M_-} F_u(\Lambda)\gamma_0^{u_{n+1}}x^u,
\end{equation}
where
\begin{equation}
F_u(\Lambda) = (\Lambda_1\cdots\Lambda_{\mu+1})^{-1} 
\sum_{\substack{l\in E\\ {\bf b}+\sum_{j=1}^N l_j{\bf a}^+_j = u}} (-1)^{\sum_{j=1}^{\mu+1}l_j}\frac{\prod_{j=1}^{\mu+1} (-l_j)!}{\prod_{j=\mu+2}^N l_j!} \prod_{j=1}^N \Lambda_j^{l_j}.
\end{equation}

It follows from the definition of $\zeta(t)$ that for $j=1,\dots,\mu+1$
\[ \frac{\partial}{\partial\Lambda_j}\zeta(\gamma_0\Lambda_jx^{{\bf a}_j^+}) = \gamma_0x^{{\bf a}_j^+} \zeta(\gamma_0\Lambda_jx^{{\bf a}_j^+}) - \frac{1}{\Lambda_j}. \]
A straightforward calculation then gives
\begin{equation}
\frac{\partial}{\partial\Lambda_j}F(\Lambda,x) = \delta_-\big(\gamma_0x^{{\bf a}_j^+}F(\Lambda,x)\big)
\end{equation}
for $j=1,\dots,\mu+1$.  Equivalently, for $u\in M_-$ we have by (4.3)
\begin{equation} \frac{\partial}{\partial\Lambda_j}F_u(\Lambda) = F_{u-{\bf a}_j^+}(\Lambda). 
\end{equation}
More generally, if $l_1,\dots,l_N$ are nonnegative integers, then
\begin{equation}
\prod_{j=1}^N \bigg(\frac{\partial}{\partial\Lambda_j}\bigg)^{l_j} F_u(\Lambda) = F_{u-\sum_{j=1}^N l_j{\bf a}_j^+}(\Lambda).
\end{equation}
In particular we have from the definition of the box operators
\begin{equation}
\Box_l\big(F_u(\Lambda)\big) = 0\quad\text{for all $l\in L$ and all $u\in M_-$.} 
\end{equation}
It is immediate from (4.4) that $F_u(\Lambda)$ satisfies the Euler operators (1.12) with $\beta=u$, hence by (4.8) the series $F_u(\Lambda)$ satisfies the $A$-hypergeometric system with parameter $\beta = u$.

Comparing (4.4) with (1.15), one sees that
\begin{equation}
 F_{\bf b}(\Lambda) = (\Lambda_1\cdots\Lambda_{\mu+1})^{-1}F(\Lambda), 
\end{equation}
a series which we noted in the Introduction has integer coefficients.  

\begin{lemma}
For all $u\in M_-$, the series $F_u(\Lambda)$ given by $(4.4)$ has integer coefficients.
\end{lemma}

\begin{proof}
Enlarge the set $\{x^{{\bf a}_j}\}_{j=1}^N$ by adding additional monomials $\{x^{{\bf a}_j}\}_{j=N+1}^{\tilde{N}}$ so that $\{x^{{\bf a}_j}\}_{j=1}^{\tilde{N}}$ consists of all monomials of degree~$d$ in $x_0,\dots,x_n$.  As in~(4.2) and~(4.3) we define
\[ \tilde{F}(\Lambda,x) = \delta_-\bigg(\prod_{j=1}^{\mu+1} \zeta(\gamma_0\Lambda_jx^{{\bf a}_j^+}) \prod_{j=\mu+2}^{\tilde{N}} \exp(\gamma_0\Lambda_jx^{{\bf a}_j^+})\bigg) \]
and set
\[ \tilde{F}(\Lambda,x) = \sum_{u\in \tilde{M}_-} \tilde{F}_u(\Lambda)\gamma_0^{u_{n+1}}x^u, \]
where $\tilde{M}\subseteq{\mathbb Z}^{n+2}$ denotes the abelian group generated by the set $\{ ({\bf a}_j,1)\}_{j=1}^{\tilde{N}}$ and $\tilde{M}_- = \tilde{M}\cap({\mathbb Z}_{<0})^{n+2}$.  The argument that proved (4.7) shows that if $l_1,\dots,l_{\tilde{N}}$ are nonnegative integers, then
\[ \prod_{j=1}^{\tilde{N}} \bigg(\frac{\partial}{\partial\Lambda_j}\bigg)^{l_j} \tilde{F}_u(\Lambda) = \tilde{F}_{u-\sum_{j=1}^N l_j{\bf a}_j^+}(\Lambda). \]
Note that for $u\in M_-$, the series $F_u(\Lambda)$ is obtained from the series $\tilde{F}_u(\Lambda)$ by setting $\Lambda_j=0$ for $j=N+1,\dots,\tilde{N}$.  To prove the lemma, it thus suffices to prove that $\tilde{F}_u(\Lambda)$ has integer coefficients for all $u\in\tilde{M}_-$.  

Every monomial in $x_0,\dots,x_n$ of degree divisible by $d$ is a product of monomials of degree $d$.  In particular, if $x^v$ is such a monomial which is divisible by $x_0\cdots x_n$, then one can write
\[ x^v = x^{{\bf a}_1}\cdots x^{{\bf a}_{\mu+1}} \prod_{j=1}^{\tilde{N}} x^{l_j{\bf a}_j} \]
for some nonnegative integers $l_1,\dots,l_{\tilde{N}}$.  It follows from this that every $u\in \tilde{M}_-$ can be written in the form
\[ u = {\bf b}-\sum_{j=1}^{\tilde{N}} l_j{\bf a}_j^+ \]
for some nonnegative integers $l_1,\dots,l_{\tilde{N}}$.  We thus have
\begin{equation}
\prod_{j=1}^{\tilde{N}}\bigg(\frac{\partial}{\partial\Lambda_j}\bigg)^{l_j} \tilde{F}_{\bf b}(\Lambda) = \tilde{F}_u(\Lambda).
\end{equation}
The series $\tilde{F}_{\bf b}(\Lambda)$ has integer coefficients by \cite[Proposition~5.2]{AS}.  It now follows from (4.11) that $\tilde{F}_u(\Lambda)$ also has integer coefficients.
\end{proof}

We can improve the conclusion of Lemma 4.10.  Fix $u\in M_-$.  There are finitely many $N$-tuples $(k_1,\dots,k_N)\in{\mathbb N}^N$ such that
\begin{equation}
u + \sum_{j=1}^N k_j{\bf a}_j^+\in M_-.
\end{equation}
Let $K_u$ be the least common multiple of the integers $\prod_{j=1}^N k_j!$ over all $(k_1,\dots,k_N)\in{\mathbb N}^N$ satisfying (4.12).
\begin{lemma}
For $u\in M_-$, all coefficients of the series $F_u(\Lambda)$ are divisible by $K_u$. 
\end{lemma}

\begin{proof}
Let $(k_1,\dots,k_N)\in{\mathbb N}^N$ satisfy (4.12) and put
\[ w = u + \sum_{j=1}^N k_j{\bf a}_j^+\in M_-. \]
It follows from (4.7) that
\[ \prod_{j=1}^N \bigg(\frac{\partial}{\partial\Lambda_j}\bigg)^{k_j} F_w(\Lambda) = F_u(\Lambda). \]
By Lemma 4.10 $F_w(\Lambda)$ has integer coefficients, so an elementary calculation shows that the coefficients of $F_u(\Lambda)$ are divisible by $\prod_{j=1}^N k_j!$.  
\end{proof}

Although the relevant hypergeometric functions appear as coefficients in the series $F(\Lambda,x)$, it is necessary for our proof of Theorem~1.16 to work with a related series which satisfies better $p$-adic estimates.  Define $G(\Lambda,x)$ to be
\begin{multline}
G(\Lambda,x) = \delta_-\big(F(\Lambda,x)\hat{\theta}_1(\Lambda,x)\big)= \\
\delta_-\bigg(\bigg(\prod_{j=1}^{\mu+1} \zeta(\gamma_0\Lambda_jx^{{\bf a}_j^+})\hat{\theta}_1(\Lambda_jx^{{\bf a}_j^+})\bigg)\bigg(\prod_{j=\mu+2}^N \hat{\theta}(\Lambda_jx^{{\bf a}_j^+})\bigg)\bigg).
\end{multline}
If we set
\begin{equation} 
G(\Lambda,x) = \sum_{u\in M_-} G_u(\Lambda)\gamma_0^{u_{n+1}}x^u, 
\end{equation}
then we have from (3.12) and (4.3) that
\begin{equation}
G_u(\Lambda) = \sum_{\substack{u^{(1)}\in M_-,u^{(2)}\in{\mathbb N}A\\ u^{(1)} + u^{(2)} = u}}F_{u^{(1)}}(\Lambda)\hat{\theta}_{1,u^{(2)}}(\Lambda).
\end{equation}

Let $K_{u^{(1)}}$ be defined as in Lemma 4.13.  By (3.13) we have
\begin{multline}
G_u(\Lambda) =  \\ 
\sum_{\substack{u^{(1)}\in M_-,u^{(2)}\in{\mathbb N}A\\ u^{(1)} + u^{(2)} = u}}K_{u^{(1)}}^{-1}F_{u^{(1)}}(\Lambda)\sum_{\substack{k_1,\dots,k_N\in{\mathbb N}\\ \sum_{j=1}^N k_j{\bf a}^+_j = u^{(2)}}} \bigg(\prod_{j=1}^N \hat{\theta}_{1,k_j}\bigg)\frac{K_{u^{(1)}}}{\prod_{j=1}^N k_j!} \Lambda_1^{k_1}\cdots\Lambda_N^{k_N}.
\end{multline}
The series $K_{u^{(1)}}^{-1}F_{u^{(1)}}(\Lambda)$ has integral coefficients by Lemma 4.13 and the ratio $\frac{K_{u^{(1)}}}{\prod_{j=1}^N k_j!}$ is an integer by the definition of $K_{u^{(1)}}$.  For each $u^{(2)}\in{\mathbb N}A$ in the inner sum on the right-hand side of (4.17) we have
\begin{equation}
{\rm ord}_p\: \prod_{j=1}^N \hat{\theta}_{1,k_j}\geq \frac{\sum_{j=1}^N k_j(p-1)}{p} = \frac{u^{(2)}_{n+1}(p-1)}{p} 
\end{equation}
by (3.10).  This implies that the series on the right-hand side of (4.17) converges to a series with integral coefficients, hence
\begin{equation}
\lvert G_u(\Lambda)\rvert\leq 1 \quad\text{for all $u\in M_-$.}
\end{equation} 

By analogy with Equation (4.9) we define $G(\Lambda)\in R_0$ by
\begin{equation}
G_{\bf b}(\Lambda) = (\Lambda_1\cdots\Lambda_{\mu+1})^{-1}G(\Lambda).
\end{equation}
\begin{lemma}
We have $G(\Lambda,x)\in S$, $\lvert G(\Lambda,x)\rvert=\lvert G_{\bf b}(\Lambda)\rvert=1$, and $G(\Lambda)$ assumes unit values on ${\mathcal D}$.
\end{lemma}

\begin{proof}
The preceding calculation shows that $G(\Lambda,x)\in S$ and $\lvert G(\Lambda,x)\rvert\leq 1$.  The estimate (4.18) shows that the inner sum on the right-hand side of (4.17) has positive $p$-ordinal unless $u^{(2)} = 0$, so
\[ G_{\bf b}(\Lambda)\equiv F_{\bf b}(\Lambda)\pmod{\gamma_0}. \]
Since $F_{\bf b}(\Lambda)$ has constant term $1$, Eq.~(4.4) shows that $F_{\bf b}(\lambda)$ is a principal unit for $\lambda\in{\mathcal D}$.  It follows that $G_{\bf b}(\Lambda)$ assumes unit values on ${\mathcal D}$, so $\lvert G_{\bf b}(\Lambda)\rvert=1$.  Eq.~(4.20) now implies that $G(\Lambda)$ is an invertible element of $R_0$ and that $G(\Lambda)$ assumes unit values on ${\mathcal D}$.
\end{proof}

\begin{theorem}
{\bf (a)} The ratio $F_u(\Lambda)/F(\Lambda)$ extends to an analytic function on ${\mathcal D}_+$ for all $u\in M_-$ if and only if the ratio $G_u(\Lambda)/G(\Lambda)$ extends to an analytic function on ${\mathcal D}_+$ for all $u\in M_-$. When either of these equivalent conditions is satisfied, the ratios $F_u(\Lambda)/G(\Lambda)$ and $G_u(\Lambda)/F(\Lambda)$ also extend to ${\mathcal D}_+$ for all $u\in M_-$. \\
{\bf (b)} If either of the equivalent conditions of part {(a)} is satisfied, then the ratio ${\mathcal F}(\Lambda):= F(\Lambda)/F(\Lambda^p)$ extends to an analytic function on ${\mathcal D}_+$ if and only if the ratio ${\mathcal G}(\Lambda):= G(\Lambda)/G(\Lambda^p)$ extends to an analytic function on ${\mathcal D}_+$.  Furthermore, if these ratios extend, then for any $\lambda\in({\mathbb F}_q^\times)^N$ with $\bar{H}(\lambda)\neq 0$, we have
\[ \prod_{i=0}^{a-1} {\mathcal F}(\hat{\lambda}^{p^i}) = \prod_{i=0}^{a-1} {\mathcal G}(\hat{\lambda}^{p^i}), \]
where $\hat{\lambda}\in{\mathbb Q}_p(\zeta_{q-1})^N$ denotes the Teichm\"uller lifting of $\lambda$.
\end{theorem}

\begin{proof}
Suppose that the ratios $F_u(\Lambda)/F(\Lambda)$ extend to analytic functions on ${\mathcal D}_+$ for all $u\in M_-$.  Divide Equation~(4.17) by $F(\Lambda)$:
\begin{multline}
\frac{G_u(\Lambda)}{F(\Lambda)} =  \\ 
\sum_{\substack{u^{(1)}\in M_-,u^{(2)}\in{\mathbb N}A\\ u^{(1)} + u^{(2)} = u}}K_{u^{(1)}}^{-1}\frac{F_{u^{(1)}}(\Lambda)}{F(\Lambda)}\sum_{\substack{k_1,\dots,k_N\in{\mathbb N}\\ \sum_{j=1}^N k_j{\bf a}^+_j = u^{(2)}}} \bigg(\prod_{j=1}^N \hat{\theta}_{1,k_j}\bigg)\frac{K_{u^{(1)}}}{\prod_{j=1}^N k_j!} \Lambda_1^{k_1}\cdots\Lambda_N^{k_N}.
\end{multline}
Since $F(\Lambda)$ assumes unit values on ${\mathcal D}$, our earlier estimates then show that this series converges to an analytic function on ${\mathcal D}_+$ that is bounded by $1$. 

In particular, the ratio
\[ \frac{G(\Lambda)}{F(\Lambda)} = \frac{\Lambda_1\cdots\Lambda_{\mu+1}G_{\bf b}(\Lambda)}{F(\Lambda)} \]
 extends to an analytic function on ${\mathcal D}_+$.  Since $G(\Lambda)$ assumes unit values on ${\mathcal D}$, it follows that $G(\Lambda)/F(\Lambda)$ assumes unit values on ${\mathcal D}_+$, hence its reciprocal $F(\Lambda)/G(\Lambda)$ is also analytic on ${\mathcal D}_+$.  Thus the product
\[ \frac{G_u(\Lambda)}{G(\Lambda)} = \frac{G_u(\Lambda)}{F(\Lambda)}\frac{F(\Lambda)}{G(\Lambda)} \]
is analytic on ${\mathcal D}_+$.  

Now suppose that the ratios $G_u(\Lambda)/G(\Lambda)$ extend to analytic functions on ${\mathcal D}_+$.  It follows from (4.14) that
\begin{equation}
F(\Lambda,x) = \delta_-\big(G(\Lambda,x)\hat{\theta}_1(\Lambda,x)^{-1}\big).
\end{equation}
One can then argue as before, using the analogue of Equation (4.17) and applying (3.17) and (3.18).  This completes the proof of part (a).

To prove part (b), let ${\mathcal H}(\Lambda) = G(\Lambda)/F(\Lambda)$.  When the equivalent conditions of part (a) are satisfied, we showed in the proof of part (a) that the function ${\mathcal H}(\Lambda)$  and its reciprocal extend to analytic functions on ${\mathcal D}_+$ and assume unit values there.  The first assertion of part (b) then follows from the equation
\begin{equation}
\frac{G(\Lambda)}{G(\Lambda^p)} = \frac{F(\Lambda)}{F(\Lambda^p)} \frac{{\mathcal H}(\Lambda)}{{\mathcal H}(\Lambda^p)}.
\end{equation}
Since ${\mathcal H}$ is analytic on ${\mathcal D}_+$, we have ${\mathcal H}(\hat{\lambda}^{p^a}) = {\mathcal H}(\hat{\lambda})$ when $\hat{\lambda}^{p^a} = \hat{\lambda}$, so
\[ \prod_{i=0}^{a-1} \frac{{\mathcal H}(\hat{\lambda}^{p^i})}{{\mathcal H}(\hat{\lambda}^{p^{i+1}})} = 1. \]
The second assertion of part (b) now follows from Equation (4.25). 
\end{proof}

By Theorem 4.22, Theorem 1.16 is equivalent to the following statement, namely, the assertion of Theorem~1.16 with $F(\Lambda)$ replaced by $G(\Lambda)$:  
\begin{theorem}
Under hypotheses $(1.5)$ and $(1.9)$, the ratio $G(\Lambda)/G(\Lambda^p)$ extends to an analytic function ${\mathcal G}(\Lambda)$ on~${\mathcal D}_+$.  Let $\lambda\in({\mathbb F}_q^\times)^N$ and let $\hat{\lambda}\in{\mathbb Q}_p(\zeta_{q-1})^N$ be its Teichm\"uller lifting.  If $\bar{H}(\lambda)\neq 0$, then $\hat{\lambda}^{p^i}\in{\mathcal D}_+$ for $i=0,\dots,a-1$ and
\[ \rho_{\min}(\lambda) = q^\mu\prod_{i=0}^{a-1} {\mathcal G}(\hat{\lambda}^{p^i}). \]
\end{theorem}

The remainder of this article is devoted to the proofs of Proposition~1.7 and Theorem~4.26.

\section{Contraction mapping}

We construct a map $\phi$ on a certain space of formal series whose coefficients 
are $p$-adic analytic functions.  Hypothesis (1.5) will imply that $\phi$ is a contraction mapping.

Let 
\[ \xi(\Lambda,x) = \sum_{\nu\in M_-} \xi_{\nu}(\Lambda) \gamma_0^{\nu_{n+1}}x^{\nu}\in S. \]
We claim that the product $\theta(\Lambda,x)\xi(\Lambda^p,x^p)$ is well defined as a formal series in $x$.
Formally we have
\[ \theta(\Lambda,x)\xi(\Lambda^p,x^p) = \sum_{\rho\in M} \zeta_{\rho}(\Lambda)x^{\rho}, \]
where
\begin{equation}
\zeta_{\rho}(\Lambda) = \sum_{\substack{u\in{\mathbb N}A,\,\nu\in M_-\\ u+p\nu=\rho}}\gamma_0^{\nu_{n+1}}\theta_u(\Lambda)\xi_{\nu}(\Lambda^p).
\end{equation}
Since $\theta_u(\Lambda)$ is a polynomial, the product $\theta_u(\Lambda)\xi_{\nu}(\Lambda^p)$ is clearly well defined.  It follows from (3.21), (3.23), and the equality $u+p\nu=\rho$ that the coefficients of $\gamma_0^{\nu_{n+1}}\theta_u(\Lambda)$ all have $p$-ordinal at least $\big(\rho_{n+1}/(p-1)\big)-\nu_{n+1}$.  Since $\lvert\xi_{\nu}(\Lambda)\rvert$ is bounded independently of $\nu$ and there are only finitely many terms on the right-hand side of (5.1) with a given value of~$\nu_{n+1}$, the series~(5.1) converges to an element of $R_{\rho}$.  

Define for $\xi(\Lambda,x)\in S$
\begin{align*}
\alpha^*(\xi(\Lambda,x)) &= \delta_-\big(\theta(\Lambda,x)\xi(\Lambda^p,x^p)\big) \\
 &= \sum_{\rho\in M_-} \zeta_{\rho}(\Lambda)x^{\rho}.
\end{align*}
For $\rho\in M_-$, put $\eta_{\rho}(\Lambda) = \gamma_0^{-\rho_{n+1}}\zeta_{\rho}(\Lambda)$, so that
\begin{equation}
\alpha^*\big(\xi(\Lambda,x)\big) = \sum_{\rho\in M_-} \eta_{\rho}(\Lambda)\gamma_0^{\rho_{n+1}}x^{\rho}
\end{equation}
with (by (5.1))
\begin{equation}
\eta_{\rho}(\Lambda) = \sum_{\substack{u\in{\mathbb N}A,\,\nu\in M_-\\ u+p\nu=\rho}}\gamma_0^{-\rho_{n+1}+\nu_{n+1}}\theta_u(\Lambda)\xi_{\nu}(\Lambda^p).
\end{equation}

\begin{proposition}
The map $\alpha^*$ is an endomorphism of $S$ and of $S'$, and for $\xi(\Lambda,x)\in S$ we have
\begin{equation}
\lvert\alpha^*(\xi(\Lambda,x))\rvert\leq \lvert p^{\mu+1}\xi(\Lambda,x)\rvert.
\end{equation}
\end{proposition}

\begin{proof}
By (5.2), the proposition will follow from the estimate
\[ \lvert\eta_{\rho}(\Lambda)\rvert\leq \lvert p^{\mu+1}\xi(\Lambda,x)\rvert\quad\text{for all $\rho\in M_-$.} \]
Using (5.3), we see that this estimate will follow in turn from the estimate
\[ \lvert\gamma_0^{-\rho_{n+1}+\nu_{n+1}}\theta_u(\Lambda)\rvert\leq \lvert p^{\mu+1}\rvert \]
for all $u\in{\mathbb N}A$, $\nu\in M_-$, with $u+p\nu = \rho$.  From (3.21) and (3.23) we see that all coefficients of $\gamma_0^{-\rho_{n+1}+\nu_{n+1}}\theta_u(\Lambda)$ have $p$-ordinal greater than or equal to
\[ \frac{-\rho_{n+1}+\nu_{n+1}+u_{n+1}}{p-1}. \]
Since $u+p\nu=\rho$, this expression simplifies to $-\nu_{n+1}$, which is greater than or equal to $\mu+1$ because $\nu\in M_-$.  
\end{proof}

Note that the equality $-\nu_{n+1} = \mu+1$ occurs for only one point $\nu\in M_-$, namely, $\nu = (-1,\dots,-1,-\mu-1)$ ($={\bf b}$).  The following corollary is then an immediate consequence of the proof of Proposition~5.4.
\begin{corollary}
If $\xi_{\bf b}(\Lambda) = 0$, then $\lvert\alpha^*(\xi(\Lambda,x))\rvert\leq \lvert p^{\mu+2}\xi(\Lambda,x)\rvert$.
\end{corollary}

We examine the polynomial $\theta_{-(p-1){\bf b}}(\Lambda)$ to determine its relation to~$H(\Lambda)$.  Let
\[ V = \{ v=(v_1,\dots,v_N)\in{\mathbb N}^N\mid \sum_{j=1}^N v_j{\bf a}^+_j = -(p-1){\bf b}\}. \]
From (3.21) and (3.22) we have
\[ \theta_{-(p-1){\bf b}}(\Lambda) = \sum_{v\in V} \bigg(\prod_{j=1}^N \theta_{v_j}\bigg)\Lambda_1^{v_1}\cdots \Lambda_N^{v_N}. \]
Clearly $v_j\leq p-1$ for all $j$, so $\theta_{v_j} = \gamma_0^{v_j}/v_j!$.  Furthermore, $\sum_{j=1}^N v_j = (p-1)(\mu+1)$, so this formula can be written
\[ \theta_{-(p-1){\bf b}}(\Lambda) = \gamma_0^{(p-1)(\mu+1)}\sum_{v\in V} \frac{\Lambda_1^{v_1}\cdots \Lambda_N^{v_N}}{v_1!\cdots v_N!}. \]

\begin{lemma}
We have the congruence
\[ (-p)^{\mu+1}H(\Lambda)\equiv \theta_{-(p-1){\bf b}}(\Lambda)\pmod{p^{\mu+2}}. \]
\end{lemma}

\begin{proof}
This follows immediately from the definitions and congruence (3.24), which implies that $\gamma_0^{(p-1)(\mu+1)}\equiv (-p)^{\mu+1}\pmod{p^{\mu+2}}$.  
\end{proof}

\begin{corollary}
The Laurent polynomial $(\Lambda_1\cdots\Lambda_{\mu+1})^{-(p-1)}\theta_{-(p-1){\bf b}}(\Lambda)$ is an invertible element of $R_0'$ with
\[ \lvert(\Lambda_1\cdots\Lambda_{\mu+1})^{-(p-1)}\theta_{-(p-1){\bf b}}(\Lambda)\rvert = \lvert p^{\mu+1}\rvert. \]
\end{corollary}

\begin{proof}
It is an invertible element of $R_0'$ by Lemma 5.7.  The assertion about the norm follows from the fact that all coefficients are divisible by $\gamma_0^{(p-1)(\mu+1)}$ and the constant term equals $(\gamma_0^{p-1}/(p-1)!)^{\mu+1}$.  
\end{proof}

We observed earlier that an element 
\[ \xi_0(\Lambda) = \sum_{\nu\in L'} c_\nu\Lambda^\nu\in R_0 \]
is invertible in the ring $R_0$ if and only if $|c_0|\geq |c_\nu|$ for all $\nu$.  Note that if $\xi_{\bf b}(\Lambda)\in R_{\bf b}$ then $\Lambda_1\cdots\Lambda_{\mu+1}\xi_{\bf b}(\Lambda)\in R_0$.  

Let $\xi(\Lambda,x)\in S$ and let $\eta(\Lambda,x) = \alpha^*\big(\xi(\Lambda,x)\big)$.  Then $\eta(\Lambda,x)$ is given by the right-hand side of (5.2), and by (5.3) we have
\begin{align}
 \eta_{\bf b}(\Lambda) &= \sum_{\substack{u\in{\mathbb N}A,\,\nu\in M_-\\ u+p\nu ={\bf b}}} \gamma_0^{\mu+1+\nu_{n+1}}\theta_u(\Lambda)\xi_{\nu}(\Lambda^p) \nonumber \\
&= \theta_{-(p-1){\bf b}}(\Lambda)\xi_{\bf b}(\Lambda^p) + \sum_{\substack{u\in{\mathbb N}A,\, \nu\in M_-\\ u+p\nu ={\bf b}\\ -\nu_{n+1}\geq \mu+2}}\gamma_0^{\mu+1+\nu_{n+1}}\theta_u(\Lambda)\xi_{\nu}(\Lambda^p).
\end{align}

\begin{lemma}
Suppose that $(\Lambda_1\cdots\Lambda_{\mu+1})\xi_{\bf b}(\Lambda)$ is an invertible element of $R_0$ (resp.~$R_0'$) and $\lvert\xi_{\bf b}(\Lambda)\rvert = \lvert\xi(\Lambda,x)\rvert$.  Then $(\Lambda_1\cdots\Lambda_{\mu+1})\eta_{\bf b}(\Lambda)$ is also an invertible element of~$R_0$ (resp.\ $R_0'$) and 
\[ \lvert\eta(\Lambda,x)\rvert = \lvert\eta_{\bf b}(\Lambda)\rvert = \lvert p^{\mu+1}\xi_{\bf b}(\Lambda)\rvert. \]
\end{lemma}

\begin{proof}
First note that
\begin{multline*}
\bigg(\prod_{j=1}^{\mu+1}\Lambda_j\bigg)\theta_{-(p-1){\bf b}}(\Lambda)\xi_{\bf b}(\Lambda^p) = \\
\bigg(\bigg(\prod_{j=1}^{\mu+1}\Lambda_j\bigg)^{-(p-1)}\theta_{-(p-1){\bf b}}(\Lambda)\bigg)\cdot\bigg( \bigg(\prod_{j=1}^{\mu+1}\Lambda_j\bigg)^p\xi_{\bf b}(\Lambda^p)\bigg), 
\end{multline*}
where the right-hand side is a product of two invertible elements by Corollary 5.8.  Also by Corollary 5.8, it has norm
\begin{equation}
\lvert p^{\mu+1}\xi_{\bf b}(\Lambda)\rvert = \lvert p^{\mu+1}\xi(\Lambda,x)\rvert. 
\end{equation}
From (3.21), (3.23), and the condition $u+p\nu = {\bf b}$ it follows that all terms in the summation on the last line of (5.9) have $p$-ordinal greater than or equal to
\begin{equation}
 \frac{\mu+1+\nu_{n+1}}{p-1} +\frac{-p\nu_{n+1}-\mu-1}{p-1} = -\nu_{n+1}\geq \mu+2. 
\end{equation}
Estimates (5.11), (5.12), and Corollary 5.8, combined with Equation (5.9), show that the function $(\prod_{j=1}^{\mu+1}\Lambda_j)\eta_{\bf b}(\Lambda)$ is invertible and that 
\[ \lvert\eta_{\bf b}(\Lambda)\rvert = \lvert p^{\mu+1}\xi_{\bf b}(\Lambda)\rvert = \lvert p^{\mu+1}\xi(\Lambda,x)\rvert. \]
Using Proposition 5.4, we then have
\[ \lvert p^{\mu+1}\xi(\Lambda,x)\rvert \geq \lvert \eta(\Lambda,x)\rvert = \sup_u \{\lvert\eta_u(\Lambda)\rvert\} \geq \lvert\eta_{\bf b}(\Lambda)\rvert  = \lvert p^{\mu+1}\xi(\Lambda,x)\rvert. \]
In particular, we get $|\eta(\Lambda,x)| = |\eta_{\bf b}(\Lambda)|$.
\end{proof}

Put
\[ T = \{\xi(\Lambda,x)\in S\mid \text{$\xi_{\bf b}(\Lambda)=(\Lambda_1\cdots\Lambda_{\mu+1})^{-1}$ and $\lvert\xi(\Lambda,x)\rvert = 1$}\} \]
and put $T' = T\cap S'$.  It follows from Lemma 5.10 that if $\xi(\Lambda,x)\in T$, then $\Lambda_1\cdots\Lambda_{\mu+1}\eta_{\bf b}(\Lambda)$ is invertible.  We may thus define for $\xi(\Lambda,x)\in T$
\[ \phi\big(\xi(\Lambda,x)\big) = \frac{\alpha^*\big(\xi(\Lambda,x)\big)}{\Lambda_1\cdots\Lambda_{\mu+1}\eta_{\bf b}(\Lambda)}. \]
Lemma 5.10 also implies that
\[ \bigg\lvert \frac{\alpha^*\big(\xi(\Lambda,x)\big)}{\Lambda_1\cdots\Lambda_{\mu+1}\eta_{\bf b}(\Lambda)}\bigg\rvert = 1, \]
so $\phi(T)\subseteq T$ and $\phi(T')\subseteq T'$.  

\begin{proposition}
The operator $\phi$ is a contraction mapping on the complete metric space $T$.  More precisely, if $\xi^{(1)}(\Lambda,x), \xi^{(2)}(\Lambda,x)\in T$, then
\[ \big\lvert\phi\big(\xi^{(1)}(\Lambda,x)\big) - \phi\big(\xi^{(2)}(\Lambda,x)\big)\big\rvert\leq \lvert p\rvert\cdot\big\lvert \xi^{(1)}(\Lambda,x)-\xi^{(2)}(\Lambda,x)\big\rvert. \]
\end{proposition}

\begin{proof}
We have (in the obvious notation)
\begin{multline*}
\phi\big(\xi^{(1)}(\Lambda,x)\big)-\phi\big(\xi^{(2)}(\Lambda,x)\big) = \frac{\alpha^*\big(
\xi^{(1)}(\Lambda,x)\big)}{\Lambda_1\cdots\Lambda_{\mu+1}\eta^{(1)}_{\bf b}(\Lambda)} - \frac{\alpha^*\big(\xi^{(2)}(\Lambda,x)
\big)}{\Lambda_1\cdots\Lambda_{\mu+1}\eta^{(2)}_{\bf b}(\Lambda)} \\
 = \frac{\alpha^*\big(\xi^{(1)}(\Lambda,x)-\xi^{(2)}(\Lambda,x)\big)}{\Lambda_1\cdots\Lambda_{\mu+1}\eta^{(1)}_{\bf b}(\Lambda)}
  - \alpha^*\big(\xi^{(2)}(\Lambda,x)\big)\frac{\eta^{(1)}_{\bf b}(\Lambda) - 
\eta^{(2)}_{\bf b}(\Lambda)}{\Lambda_1\cdots\Lambda_{\mu+1}\eta^{(1)}_{\bf b}(\Lambda)\eta^{(2)}_{\bf b}(\Lambda)}.
\end{multline*}
By Corollary 5.6 and Lemma 5.10 we have
\[ \bigg\lvert\frac{\alpha^*\big(\xi^{(1)}(\Lambda,x)-\xi^{(2)}(\Lambda,x)\big)}{\Lambda_1\cdots\Lambda_{\mu+1}\eta^{(1)}_{\bf b}(\Lambda)}
\bigg\rvert\leq \lvert p\rvert\cdot
\big\lvert \xi^{(1)}(\Lambda,x)-\xi^{(2)}(\Lambda,x)\big\rvert. \]
Since $\eta^{(1)}_{\bf b}(\Lambda)-\eta^{(2)}_{\bf b}(\Lambda)$ is the coefficient of 
$x^{\bf b}$ in $\alpha^*\big(\xi^{(1)}(\Lambda,x)-\xi^{(2)}(\Lambda,x)\big)$, we have
\begin{align*} 
\lvert\eta^{(1)}_{\bf b}(\Lambda)-\eta^{(2)}_{\bf b}(\Lambda)\rvert & \leq 
\big\lvert\alpha^*\big(\xi^{(1)}(\Lambda,x)-\xi^{(2)}(\Lambda,x)\big)\big\rvert \\ 
& \leq  \lvert p^{\mu+2}\rvert\cdot\big\lvert \xi^{(1)}(\Lambda,x)-\xi^{(2)}(\Lambda,x)\big\rvert 
\end{align*}
by Corollary 5.6.  We have $\lvert\eta^{(1)}_{\bf b}(\Lambda)\eta^{(2)}_{\bf b}(\Lambda)\rvert=\lvert p^{2\mu+2}\rvert$ by Lemma~5.10, so by (5.5)
\[ \bigg\lvert \alpha^*\big(\xi^{(2)}(\Lambda,x)\big)\frac{\eta^{(1)}_{\bf b}(\Lambda) - 
\eta^{(2)}_{\bf b}(\Lambda)}{\Lambda_1\cdots\Lambda_{\mu+1}\eta^{(1)}_{\bf b}(\Lambda)\eta^{(2)}_{\bf b}(\Lambda)}\bigg\rvert
\leq \lvert p\rvert \cdot\big\lvert \xi^{(1)}(\lambda,x)-\xi^{(2)}(\lambda,x)\big\rvert. \]
This establishes the proposition.
\end{proof} 

By a well-known theorem, Proposition 5.13 implies that $\phi$ has a unique fixed point in $T$.  And since $\phi$ is stable on $T'$, that fixed point must lie in $T'$.  This fixed point of $\phi$ is related to a certain eigenvector of $\alpha^*$. 

\begin{theorem}
We have $\alpha^*\big(G(\Lambda,x)\big) = p^{\mu+1} G(\Lambda,x)$.
\end{theorem}

The proof of Theorem 5.14 will be given in the next section.  In the remainder of this section, we use Proposition 5.13 and Theorem 5.14 to prove that $G(\Lambda)/G(\Lambda^p)\in R_0'$.  This will establish the first sentence of Theorem~4.26.  Note that Lemma 4.21 implies that $G(\Lambda,x)/G(\Lambda)\in T$.  
 
\begin{proposition}
The unique fixed point of $\phi$ in $T$ is $G(\Lambda,x)/G(\Lambda)$, hence $G(\Lambda,x)/G(\Lambda)\in T'$.  In particular,
\[ \frac{G_u(\Lambda)}{G(\Lambda)}\in R_u'\quad\text{for all $u\in M_-$.} \]
\end{proposition}

\begin{proof}
We have
\begin{align}
\alpha^*\bigg(\frac{G(\Lambda,x)}{G(\Lambda)}\biggr)  &= \frac{\alpha^*\big(G(\Lambda,x)\big)}{G(\Lambda^p)} \nonumber \\
&= \bigg(\frac{p^{\mu+1}G(\Lambda)}
{G(\Lambda^p)}\bigg) \frac{G(\Lambda,x)}{G(\Lambda)},
\end{align}
where the second equality follows from Theorem 5.14.  By the definition of $\phi$, this implies the result.
\end{proof}

\begin{corollary}
With the above notation, $G(\Lambda)/G(\Lambda^p)\in R_0'$.
\end{corollary}

\begin{proof}
Since $\alpha^*$ is stable on $S'$, Proposition 5.15 implies that the right-hand side of~(5.16) lies in~$S'$.  Since the coefficient of $\gamma_0^{-\mu-1}x^{\bf b}$ on the right-hand side of~(5.16) 
is $p^{\mu+1}(\Lambda_1\cdots\Lambda_{\mu+1})^{-1}G(\Lambda)/G(\Lambda^p)$, the result follows.
\end{proof}

\section{Proof of Theorem 5.14}

Consider the space of formal series
\[ C=\bigg\{ \xi = \sum_{i=0}^\infty c_i i! \gamma_0^{-i-1}t^{-i-1}\mid \text{$\{c_i\}_{i=0}^\infty$ is bounded}\bigg\}. \]
Recall that $\delta_-$ is the truncation operator on series: 
\[ \delta_-\bigg(\sum_{i=-\infty}^\infty d_it^{-i-1}\bigg) = \sum_{i=0}^\infty d_it^{-i-1}. \]
 
\begin{lemma}
The map $\delta_-\circ\hat{\theta}_1(t)$ is an isomorphism of $C$ with itself.  The inverse isomorphism is $\delta_-\circ\hat{\theta}_1(t)^{-1}$.  (We use $\hat{\theta}_1(t)$ (resp.\ $\hat{\theta}_1(t)^{-1}$) as operator to mean multiplication by $\hat{\theta}_1(t)$ (resp.\ $\hat{\theta}_1(t)^{-1}$).)
\end{lemma}

\begin{proof}
Let $\xi = \sum_{j=0}^\infty c_j j! \gamma_0^{-j-1}t^{-j-1}\in C$ and let $k$ be a nonnegative integer.  To simplify the estimate, assume that the $c_j$ are bounded by 1.  The coefficient of $t^{-k-1}$ in the product $\hat{\theta}_1(t)\xi$ is
\[ \sum_{i-j-1=-k-1} c_j j! \gamma_0^{-j-1}\frac{\hat{\theta}_{1,i}}{i!}\gamma_0^i= \bigg(\sum_{i=0}^\infty \hat{\theta}_{1,i}c_{i+k}\frac{(i+k)!}{i!k!}\bigg) k!\gamma_0^{-k-1}. \]
We have by (3.10)
\[ {\rm ord}_p\:\hat{\theta}_{1,i}c_{i+k}\frac{(i+k)!}{i!k!} \geq \frac{i(p-1)}{p} -\frac{s_{i+k} +s_i+s_k}{p-1}\geq \frac{i(p-1)}{p}. \]
This shows that the series $\sum_{i=0}^\infty\hat{\theta}_{1,i}c_{i+k}(i+k)!/(i!k!)$ converges and is bounded by~$1$, hence $\delta_-\circ\hat{\theta}_1(t)$ maps $C$ into itself.  Since the coefficients of the reciprocal power series $\hat{\theta}_1(t)^{-1} = \prod_{j=1}^\infty \exp(-\gamma_jt^{p^j})$ satisfy the same estimate (3.15), the same argument shows that $\delta_-\circ\hat{\theta}_1(t)^{-1}$ also maps $C$ into itself and hence is the inverse of $\delta_-\circ\hat{\theta}_1(t)$.
\end{proof}

Define an operator $D'$ on $C$ by
\begin{equation}
D'=\delta_-\circ \bigg( t\frac{d}{dt} - \sum_{j=0}^\infty \gamma_jp^jt^{p^j}\bigg) = 
\delta_-\circ \hat{\theta}(t)\circ t\frac{d}{dt}\circ \hat{\theta}(t). 
\end{equation}
\begin{proposition}
The operator $D'$ has a one-dimensional (over ${\mathbb C}_p$) kernel as operator on the space $C$.
\end{proposition}

\begin{proof}
If $\xi\in C$ is a solution of $D'$, then $\delta_-(\hat{\theta}_1(t)^{-1}\xi)$ lies in $C$ by Lemma~6.1 and is a solution of the operator
\begin{equation} 
\delta_-\circ\bigg(t\frac{d}{dt} - \gamma_0 t\bigg) = \delta_-\circ \exp(\gamma_0 t)\circ t\frac{d}{dt} \circ\exp(-\gamma_0 t). 
\end{equation}
Conversely, if $\xi\in C$ is a solution of (6.4), then $\delta_-(\hat{\theta}_1(t)\xi)$ lies in $C$ and is a solution of $D'$.  Thus it suffices to show that (6.4) has a unique solution (up to scalars) in~$C$.  Applying the operator (6.4) to $\xi = \sum_{i=0}^\infty c_i i! \gamma_0^{-i-1}t^{-i-1}\in C$ gives
\[ \sum_{i=0}^{\infty} (-c_i-c_{i+1})(i+1)!\gamma_0^{-i-1}t^{-i-1}, \]
from which it is clear that the solutions of (6.4) in $C$ are scalar multiples of
\begin{equation}
q(t) := \sum_{i=0}^\infty (-1)^ii!\gamma_0^{-i-1}t^{-i-1}. 
\end{equation}
\end{proof}

Define
\begin{equation}
Q(t) = \delta_-(\hat{\theta}_1(t)q(t)) = \sum_{i=0}^\infty Q_i i! \gamma_0^{-i-1}t^{-i-1}. 
\end{equation}
From Lemma 6.1 we have $Q(t)\in C$; the proof of Lemma 6.1 shows that the $Q_i$ are $p$-integral.  
From the proof of Proposition 6.3 we get the following corollary.
\begin{corollary}
The solutions of $D'$ in $C$ are the scalar multiples of $Q(t)$.
\end{corollary}

For $\xi(t)= \sum_{i=0}^\infty c_i i! \gamma_0^{-i-1}t^{-i-1}\in C$ define $\alpha'(\xi)$ to be
\[ \alpha'(\xi) = \delta_-(\theta(t)\xi(t^p)). \]

\begin{proposition}
The operator $\alpha'$ maps $C$ into itself.
\end{proposition}

\begin{proof}
For $k\geq 0$, the coefficient of $t^{-k-1}$ in $\theta(t)\xi(t^p)$ is
\[ \sum_{\substack{i,j\geq 0\\ j-pi-p = -k-1}} \theta_jc_i i! \gamma_0^{-i-1}. \]
We may assume the $c_i$ to be $p$-integral, in which case we have the estimate
\begin{align*}
{\rm ord}_p\:\theta_jc_ii!\gamma_0^{-i-1} &\geq \frac{j}{p-1}+\frac{i-s_i}{p-1} -\frac{i+1}{p-1} \\
 &=\frac{j-s_i-1}{p-1}.
\end{align*}
Since $i$ is a linear function of $j$ ($k$ is fixed) and $s_i$ is bounded above by a positive multiple of $\log i$, this estimate shows that the series converges.  The condition $j-pi-p = -k-1$ gives $j+k=pi+(p-1)$, which implies
\[ s_{j+k} = s_i+(p-1). \]
Since $s_j+s_k\geq s_{j+k}$, we get the estimate
\[ {\rm ord}_p\:\theta_jc_ii!\gamma_0^{-i-1}\geq \frac{j-s_j+(p-1)}{p-1} - \frac{s_k+1}{p-1}. \]
The first term on the right-hand side is always $\geq 1$, which implies that we can write
\[  \sum_{\substack{i,j\geq 0\\ j-pi-p = -k-1}} \theta_jc_i i! \gamma_0^{-i-1} = pd_kk!\gamma_0^{-k-1} \]
for some $d_k$ which is $p$-integral.  This proves the proposition.
\end{proof}

\begin{proposition}
As operators on $C$ we have $D'\circ\alpha' = p\alpha'\circ D'$.
\end{proposition}

\begin{proof}
If we let $\Phi$ be the map that sends an element $\xi(t)\in C$ to $\xi(t^p)$, then we may factor $\alpha'$ as
\[ \alpha' = \delta_-\circ\hat{\theta}(t)\circ\Phi\circ\hat{\theta}(t)^{-1}. \]
Combined with the corresponding factorization of $D'$ (see Equation (6.2)), this reduces the assertion of the proposition to the obvious equality
\[ t\frac{d}{dt}\circ\Phi = p\Phi\circ t\frac{d}{dt}. \]
\end{proof}

It follows from Corollary 6.7 and Proposition 6.9 that $Q(t)$ is an eigenvector of~$\alpha'$.  More precisely, we have the following result.
\begin{proposition}
$\alpha'(Q(t)) = pQ(t)$
\end{proposition}

\begin{proof}
Let $C^*$ be the space of series
\[ C^* = \bigg\{ \eta(t) = \sum_{i=0}^\infty c_i\gamma_0^it^i\mid \text{$\{c_i\}$ is bounded}\bigg\} \]
and let $C^*_0$ be the subset consisting of those series $\eta\in C^*$ with $c_0=0$.  The differential operator
$D:=td/dt + \sum_{j=0}^\infty \gamma_jp^jt^{p^j}$ acts on $C^*$ and by \cite[Theorem~3.8]{AS0} the map
$D:C\to C^*_0$ is an isomorphism.  

Define $\psi:C^*\to C^*$ by $\psi(\sum_{i=0}^\infty c_i\gamma_0^it^i) = \sum_{i=0}^\infty c_{pi}\gamma_0^{pi}t^i$ and define $\alpha:C^*\to C^*$ to be the composition $\psi\circ\theta(t)$.  A calculation analogous to the proof of Proposition~6.9 shows that as operators on $C^*$
\begin{equation}
\alpha\circ D = pD\circ\alpha.
\end{equation}
We have a commutative diagram with exact rows
\begin{equation}
\begin{CD}
0 @>>> C^*_0 @>>> C^* @>>> {\mathbb C}_p @>>> 0 \\
@VVV @V{D}VV @V{D}VV @VVV @VVV \\
0 @>>> C^*_0 @>>{\rm id}> C^*_0 @>>> 0 @>>> 0
\end{CD}
\end{equation}
where $C^*_0\to C^*$ is the inclusion, $C^*_0\to C^*_0$ is the identity, and $C^*\to{\mathbb C}_p$ is the map ``set $t=0$.''  Since $D:C^*\to C^*_0$ is an isomorphism, the long-exact cohomology sequence associated to (6.12) implies that there is an isomorphism ${\mathbb C}_p\cong C^*_0/DC^*_0$ which identifies $1\in{\mathbb C}_p$ with the class $D(1) + DC_0^*\in C_0^*/DC_0^*$.  It is easily seen that $\alpha(1)\in 1+C_0^*$, so (6.11) implies
\begin{equation}
\alpha\big(D(1)\big) = pD\big(\alpha(1)\big) \equiv pD(1)\pmod{DC_0^*}. 
\end{equation}
It follows that the induced action of $\alpha$ on ${\mathbb C}_p\cong C_0^*/DC_0^*$ is multiplication by $p$.

Define a pairing between the spaces $C$ and $C^*_0$: for $\xi = \sum_{i=0}^{\infty} c_ii!\gamma_0^{-i-1}t^{-i-1}\in C$ and $\eta = \sum_{i=0}^\infty b_i\gamma_0^{i+1}t^{i+1}\in C^*_0$ put
\[ \langle \xi,\eta \rangle = \sum_{i=0}^\infty b_ic_ii!. \]
The series on the right-hand side converges because the $\{c_i\}$ and $\{b_i\}$ are bounded and $i!\to 0$ as $i\to\infty$.  
Note that if $u\in{\mathbb Z}_{>0}$ and $v\in{\mathbb Z}_{<0}$, then
\[ \langle t^v,D(t^u)\rangle = -\langle D'(t^v),t^u\rangle = \begin{cases} u & \text{if $u+v=0$,} \\ \gamma_jp^j & \text{if $u+v=-p^j$ for some $j$,}\\ 0 & \text{otherwise,} \end{cases} \]
which implies that 
\begin{equation}
\langle D'(\xi),\eta\rangle = -\langle \xi,D(\eta) \rangle
\end{equation}
 for $\xi\in C$ and $\eta\in C^*_0$.
A direct calculation also shows that 
\[ \langle \alpha'(t^v),t^u\rangle = \langle t^v,\alpha(t^u)\rangle = \theta_{-pv-u}, \]
which implies that 
\begin{equation}
\langle \alpha'(\xi),\eta\rangle = \langle \xi,\alpha(\eta) \rangle
\end{equation}
 for $\xi\in C$ and $\eta\in C^*_0$.  We then have
\begin{align*}
 \langle \alpha'(Q(t)),D(1)\rangle &= \langle Q(t),\alpha(D(1) \rangle \\
 &= \langle Q(t),pD(1) + \eta\rangle
\end{align*}
for some $\eta\in DC^*_0$ by (6.13).  But $\langle Q(t),DC^*_0\rangle = 0$ by (6.14) and Corollary~6.7, so we get
\[  \langle \alpha'(Q(t)),D(1)\rangle = p \langle Q(t),D(1)\rangle. \]
Since we already know that $\alpha'(Q(t))$ is a scalar multiple of $Q(t)$, the proposition will follow from this equality once we have checked that $\langle Q(t),D(1)\rangle\neq 0$.

We have $D(1) = \sum_{j=0}^\infty \gamma_jp^jt^{p^j}$ and $Q(t) = \sum_{i=0}^\infty Q_ii!\gamma_0^{-i-1}t^{-i-1}$, so
\begin{equation}
\langle Q(t),D(1)\rangle = \sum_{j=0}^\infty \gamma_jp^jQ_{p^j-1} (p^j-1)!\gamma_0^{-p^j}. 
\end{equation}
We have by (3.4) and the $p$-integrality of the $Q_i$
\[ {\rm ord}_p\: \gamma_jp^jQ_{p^j-1} (p^j-1)!\gamma^{-p^j}\geq \frac{p^{j+1}}{p-1} - (j+1) +j +\frac{p^j-1-j(p-1)}{p-1}-\frac{p^j}{p-1}, \]
which simplifies to
\[ {\rm ord}_p\: \gamma_jp^jQ_{p^j-1} (p^j-1)!\gamma^{-p^j}\geq \sum_{i=0}^j (p^i-1). \]
The right-hand side of this inequality is an increasing function of $j$, positive for $j>0$, so to prove the expression (6.16) is not zero, it suffices to show that $Q_0$, the contribution to the sum on the right-hand side of (6.16) for $j=0$, is a unit.  From the definition (6.6) we compute
\[ Q_0 = \sum_{i=0}^\infty (-1)^i \hat{\theta}_{1,i}. \]
The desired assertion about $Q_0$ then follows from (3.10) and the fact that \mbox{$\hat{\theta}_{1,0} = 1$}.
\end{proof}

For $i=1,\dots,\mu+1$, consider the series
$Q(\Lambda_ix^{{\bf a}^+_i})$.  Proposition 6.10 implies that all exponents of the series
\[ \theta(\Lambda_ix^{{\bf a}^+_i})Q(\Lambda_i^px^{p{\bf a}^+_i})-pQ(\Lambda_ix^{{\bf a}^+_i}) \]
are nonnegative (we are just replacing $t$ in Proposition~6.10 by $\Lambda_i x^{{\bf a}_i^+}$).  It follows that
\begin{equation}
\delta_-\bigg( \prod_{i=1}^{\mu+1} \big( \theta(\Lambda_ix^{{\bf a}^+_i})Q(\Lambda_i^px^{p{\bf a}^+_i})-pQ(\Lambda_ix^{{\bf a}^+_i})\big)\bigg) = 0.
\end{equation}
Our choice of the set $\{ {\bf a}_i^+ \}_{i=1}^{\mu+1}$ implies that an integral linear combination $\sum_{i=1}^N l_i {\bf a}_i^+$ lies in $M_-$ only if $l_i<0$ for $i=1,\dots,\mu+1$.  This implies that when the product in (6.17) is expanded, the only terms not annihilated by $\delta_-$ are 
\[ \prod_{i=1}^{\mu+1}\theta(\Lambda_ix^{{\bf a}^+_i})Q(\Lambda_i^px^{p{\bf a}^+_i})\quad\text{and}\quad \prod_{i=1}^{\mu+1} pQ(\Lambda_ix^{{\bf a}^+_i}). \]
Furthermore, $\delta_-\big(\prod_{i=1}^{\mu+1} pQ(\Lambda_ix^{{\bf a}^+_i})\big) = \prod_{i=1}^{\mu+1} pQ(\Lambda_ix^{{\bf a}^+_i})$.   Equation (6.17) therefore implies that
\begin{equation}
\delta_-\bigg(\prod_{i=1}^{\mu+1} \theta(\Lambda_ix^{{\bf a}^+_i})Q(\Lambda_i^px^{p{\bf a}^+_i})\bigg) = p^{\mu+1} \prod_{i=1}^{\mu+1} Q(\Lambda_ix^{{\bf a}^+_i}).
\end{equation}

\begin{lemma}
We have
\[ G(\Lambda,x) = \delta_-\bigg(\bigg(\prod_{j=1}^{\mu+1} Q(\Lambda_ix^{{\bf a}^+_i})\bigg)\bigg(\prod_{j=\mu+2}^N \hat{\theta}(\Lambda_jx^{{\bf a}^+_j})\bigg)\bigg). \]
\end{lemma}

\begin{proof}
From the definitions of $F(\Lambda,x)$ and $q(t)$ we have
\[ F(\Lambda,x) = \delta_-\bigg(\prod_{j=1}^{\mu+1} q(\Lambda_jx^{{\bf a}_j^+})\prod_{j=\mu+2}^N \exp(\gamma_0\Lambda_jx^{{\bf a}_j^+})\bigg). \]
From the definitions of $G(\Lambda,x)$ and $\hat{\theta}_1(\Lambda,x)$ (Eqns.~(4.14) and~(3.11)) we get
\[ G(\Lambda,x) = \delta_-\bigg(\prod_{j=1}^{\mu+1} q(\Lambda_jx^{{\bf a}_j^+})\prod_{j=\mu+2}^N \exp(\gamma_0\Lambda_jx^{{\bf a}_j^+})\prod_{j=1}^N \hat{\theta}_1(\Lambda_jx^{{\bf a}_j^+})\bigg). \]
Using the definitions of $\hat{\theta}(t)$ and $\hat{\theta}_1(t)$ (Eqns.~(3.3) and~(3.9)), this equation may be rewritten as
\[ G(\Lambda,x) = \delta_-\bigg(\prod_{j=1}^{\mu+1} \big(q(\Lambda_jx^{{\bf a}_j^+})\hat{\theta}_1(\Lambda_jx^{{\bf a}_j^+})\big) \prod_{j=\mu+2}^N \hat{\theta}(\Lambda_jx^{{\bf a}_j^+})\bigg). \]
The assertion of Lemma 6.19 now follows from the definition of $Q(t)$ (Eq.~(6.6)).
\end{proof}

We can now prove Theorem 5.14.  First note that since $\theta(t) = \hat{\theta}(t)/\hat{\theta}(t^p)$, we have
\begin{equation}
\prod_{j=\mu+2}^N \theta(\Lambda_jx^{{\bf a}^+_j}) \prod_{j=\mu+2}^N \hat{\theta}(\Lambda_j^px^{p{\bf a}^+_j}) = \prod_{j=\mu+2}^N \hat{\theta}(\Lambda_jx^{{\bf a}^+_j}).
\end{equation}
We now compute:
\begin{multline*}
\alpha^*\big(G(\Lambda,x)\big) = \delta_-\bigg(\prod_{j=1}^N \theta(\Lambda_jx^{{\bf a}^+_j}) \delta_-\bigg(\bigg(\prod_{i=1}^{\mu+1} Q(\Lambda_i^px^{p{\bf a}^+_i})\bigg)\bigg(\prod_{j=\mu+2}^N \hat{\theta}(\Lambda_j^px^{p{\bf a}^+_j})\bigg)\bigg)\bigg) \\
 = \delta_-\bigg(\bigg( \prod_{i=1}^{\mu+1} \theta(\Lambda_ix^{{\bf a}^+_i})Q(\Lambda_i^px^{p{\bf a}^+_i})\bigg)\bigg(\prod_{j=\mu+2}^N \theta(\Lambda_jx^{{\bf a}^+_j}) \prod_{j=\mu+2}^N \hat{\theta}(\Lambda_j^px^{p{\bf a}^+_j}) \bigg)\bigg) \\
 = p^{\mu+1}\delta_-\bigg(\bigg(\prod_{i=1}^{\mu+1} Q(\Lambda_ix^{{\bf a}^+_i})\bigg)\bigg(\prod_{j=\mu+2}^N \hat{\theta}(\Lambda_jx^{{\bf a}^+_j})\bigg)\bigg) = p^{\mu+1}G(\Lambda,x),
\end{multline*}
where the first equality follows from Lemma 6.19, the next-to-last equality follows from Equations~(6.18) and~(6.20), and the last equality follows from Lemma~6.19.

\section{Zeta functions}  

Let $f_\lambda(x_0,\dots,x_n)$ be as defined in the Introduction.  We associate to $f_\lambda$ exponential sums
\[ S_\lambda(m) = \sum_{x\in{\mathbb A}^{n+2}({\mathbb F}_{q^m})} \Psi \big({\rm  Tr}_{{\mathbb F}_{q^m}/{\mathbb F}_p}(x_{n+1}f_\lambda(x_0,\dots,x_n))\big), \]
where $\Psi:{\mathbb F}_p\to{\mathbb Q}_p(\zeta_p)^\times$ is the additive character satisfying
\[ \Psi(1) \equiv 1+\gamma_0\pmod{\gamma_0^2}. \]
We denote the corresponding $L$-function by $L_\lambda(t)$: 
\[ L_\lambda(t) = \exp\bigg(\sum_{m=1}^\infty S_\lambda(m)\frac{t^m}{m}\bigg). \]
We recall the relationship (\cite[Equation (2.3)]{AS1}) between $L_\lambda(t)$ and the rational function $P_\lambda(t)$ defined in the Introduction:
\begin{equation}
L_\lambda(t)^{(-1)^{n+1}} = (1-q^{n+1}t)^{(-1)^n}\frac{P_\lambda(qt)}{P_\lambda(q^2t)}.
\end{equation}

We shall first prove Proposition 1.7 and then prove the last assertion of Theorem~4.26. 
We begin by reviewing the expression for $L_\lambda(t)$ that comes from Dwork's trace formula (\cite[Section 2]{AS1}).
For $s\in{\mathbb Z}$, let $L_s$ be the space of series 
\[ L_s = \bigg\{\sum_{u \in{\mathbb N}^{n+2}} c_u\gamma_0^{pu_{n+1}}x^u\mid\text{$\sum_{i=0}^n u_i-du_{n+1}=s$, $c_u\in{\mathbb C}_p$, and $\{c_u\}$ is bounded}\bigg\}. \]
For a subset $I=\{i_1,\dots,i_k\}\subseteq\{0,\dots,n+1\}$, we define
\[ L_I = \begin{cases} L_{-k} & \text{if $n+1\not\in I$,} \\ L_{d-k+1} & \text{if $n+1\in I$.} \end{cases} \]
We construct a de Rham-type complex as follows.  For $k=0,\dots,n+1$, let
\[ \Omega^k = \bigoplus_{0\leq i_1<\dots<i_k\leq n+1} L_{\{i_1,\dots,i_k\}}\,dx_{i_1}\cdots dx_{i_k} \]
Define $d:\Omega^k\to\Omega^{k+1}$ by
\[ d(\xi\,dx_{i_1}\cdots dx_{i_k}) = \sum_{i=0}^{n+1} \frac{\partial\xi}{\partial x_i}\,dx_i dx_{i_1}\cdots dx_{i_k} \]
for $\xi\in L_{\{i_1,\dots,i_k\}}$.  Define $\hat{f}_\lambda$ to be the Teichm\"uller lifting of $f_\lambda$:
\[ \hat{f}_\lambda(x_0,\dots,x_{n+1}) = \sum_{j=1}^N \hat{\lambda}_jx^{{\bf a}^+_j}\in{\mathbb Q}_p(\zeta_{q-1})[x_0,\dots,x_n]. \]
Set 
\[ h = \sum_{j=0}^\infty \gamma_jx_{n+1}^{p^j}\hat{f}^{\sigma^j}(x^{p^j}), \]
where
\[ \hat{f}^\sigma(x^p) = \sum_{j=1}^N \hat{\lambda}_j^px^{p{\bf a}^+_j}, \]
and note that $dh\in\Omega^1$.  We observe that in general, if $\omega_1\in\Omega^{k_1}$ and $\omega_2\in\Omega^{k_2}$, then $\omega_1\wedge\omega_2\in\Omega^{k_1+k_2}$.  
Let $D:\Omega^k \to\Omega^{k+1}$ be defined by
\[ D(\omega) = d\omega + dh\wedge\omega. \]
This gives a complex $(\Omega^\bullet,D)$.

We define the Frobenius operator on this complex.  From Equation (3.19) we have
\begin{equation}
\theta(\hat{\lambda},x) = \prod_{j=1}^N \theta(\hat{\lambda}_jx^{{\bf a}^+_j}). 
\end{equation}
We shall also need to consider the series $\theta_0(\hat{\lambda},x)$ defined by
\begin{equation}
\theta_0(\hat{\lambda},x) = \prod_{i=0}^{a-1} \prod_{j=1}^N \theta\big((\hat{\lambda}_jx^{{\bf a}^+_j})^{p^i}\big) = \prod_{i=0}^{a-1}\theta(\hat{\lambda}^{p^i},x^{p^i}).
\end{equation}
Define an operator $\psi$ on formal power series by
\begin{equation}
\psi\bigg(\sum_{u\in{\mathbb N}^{n+2}} c_ux^u\bigg) = \sum_{u\in{\mathbb N}^{n+2}} c_{pu}x^u.
\end{equation}
Denote by $\alpha_{\hat{\lambda}}$ the composition
\[ \alpha_{{\hat{\lambda}}} := \psi^a\circ\text{``multiplication by $\theta_0(\hat{\lambda},x)$.''} \]

We define a map $\alpha_{\hat{\lambda},\bullet}:\Omega^\bullet \to\Omega^\bullet$ by additivity and the formula
\begin{equation}
\alpha_{\hat{\lambda},k}(\xi\, dx_{i_1}\cdots dx_{i_k}) = \frac{q^{n+2-k}}{x_{i_1}\cdots x_{i_k}} \alpha_{\hat{\lambda}}(x_{i_1}\cdots x_{i_k}\xi)\,dx_{i_1}\cdots dx_{i_k},
\end{equation}
when $\xi\in L_{\{i_1,\dots,i_k\}}$.  Note that in this case $x_{i_1}\cdots x_{i_k}\xi$ and $\alpha_{\hat{\lambda}}(x_{i_1}\cdots x_{i_k}\xi)$ lie in~$L_0$.
The map $\alpha_{\hat{\lambda},\bullet}$ is a map of complexes and by the Dwork trace formula (as formulated by Robba, see \cite[Section 2]{AS1}) we have
\begin{equation}
L_{\lambda}(t) = \prod_{k=0}^{n+2} \det(I-t\alpha_{\hat{\lambda},k}\mid \Omega^k)^{(-1)^{k+1}}. 
\end{equation}
The factors on the right-hand side of (7.6) are $p$-adic entire functions.  

We shall combine (7.1) and (7.6) to get a formula for $P_\lambda(qt)$.  First of all, for $I=\{i_1,\dots,i_k\}\subseteq\{0,1,\dots,n+1\}$, let $L_0^I\subseteq L_0$ be the image of $L_I\,dx_{i_1}\cdots dx_{i_k}$ under the map $\phi$ defined by
\[ \xi\,dx_{i_1}\cdots dx_{i_k}\to x_{i_1}\cdots x_{i_k}\xi. \]
We have a commutative diagram
\[ \begin{CD}
L_I\,dx_{i_1}\cdots dx_{i_k} @>\phi>> L^I_0 \\
@V\alpha_{\hat{\lambda},k}VV @VVq^{n+2-k}\alpha_{\hat{\lambda}}V \\
L_I\,dx_{i_1}\cdots dx_{i_k} @>\phi>> L^I_0,
\end{CD} \]
in which the horizontal arrows are isomorphisms, hence there is a product decomposition
\begin{equation}
\det(I-t\alpha_{\hat{\lambda},k}\mid \Omega^k) = \prod_{|I|=k} \det(I-q^{n+2-k}t\alpha_{\hat{\lambda}}\mid L^I_0).
\end{equation}
Combining this with (7.6) gives
\begin{equation}
L_\lambda(t) = \prod_{I\subseteq\{0,1,\dots,n+1\}} \det(I-q^{n+2-|I|}t\alpha_{\hat{\lambda}}\mid L_0^I)^{(-1)^{|I|+1}}.
\end{equation}

Note that $x^u\in L_0^I$ if and only if $\sum_{i=0}^n u_i = du_{n+1}$ and $u_i>0$ for $i\in I$.  Suppose that $I\subseteq\{0,1,\dots,n\}$ and $I\neq\emptyset$.  If $x^u\in L_0^I$ then $u_{n+1}>0$ also, hence $L_0^I = L_0^{I\cup\{n+1\}}$.   It follows that for such $I$ we have
\begin{equation}
 \det(I-q^{n+1-|I|}t\alpha_{\hat{\lambda}}\mid L_0^I)=\det(I-q^{n+2-|I\cup\{n+1\}|}t\alpha_{\hat{\lambda}}\mid L_0^{I\cup\{n+1\}}).
\end{equation}
We can therefore rewrite (7.8) as
\begin{multline}
L_\lambda(t) = \frac{\det(I-q^{n+1}t\alpha_{\hat{\lambda}}\mid L_0^{\{n+1\}})}{\det(I-q^{n+2}t\alpha_{\hat{\lambda}}\mid L_0^{\emptyset})} \\
\cdot\prod_{\emptyset\neq I\subseteq\{0,1,\dots,n\}} 
\bigg(\frac{\det(I-q^{n+2-|I|}t\alpha_{\hat{\lambda}}\mid L_0^I)}{\det(I-q^{n+1-|I|}t\alpha_{\hat{\lambda}}\mid L_0^I)}\bigg)^{(-1)^{|I|+1}}.
\end{multline}
We examine the first quotient on the right-hand side of (7.10) more closely.  It is easy to see that the quotient $L_0^{\emptyset}/L_0^{\{n+1\}}$ is one-dimensional, spanned by the constant $1$, and that $\alpha_{\hat{\lambda}}$ acts on this quotient as the identity map.  We therefore have
\[ \det(I-q^{n+1}t\alpha_{\hat{\lambda}}\mid L_0^{\emptyset}) = (1-q^{n+1}t)\det(I-q^{n+1}t\alpha_{\hat{\lambda}}\mid L_0^{\{n+1\}}). \]
Equation (7.10) thus implies
\begin{multline}
L_\lambda(t)^{(-1)^{n+1}} = (1-q^{n+1}t)^{(-1)^n} \\
\cdot \frac{\prod_{I\subseteq\{0,1,\dots,n\}} 
\det(I-q^{n+2-|I|}t\alpha_{\hat{\lambda}}\mid L_0^I)^{(-1)^{n+|I|}}}{\prod_{I\subseteq\{0,1,\dots,n\}} \det(I-q^{n+1-|I|}t\alpha_{\hat{\lambda}}\mid L_0^I)^{(-1)^{n+|I|}}}.
\end{multline}
Comparing (7.1) and (7.11) now gives the desired formula:
\begin{equation}
P_\lambda(qt) = \prod_{I\subseteq\{0,1,\dots,n\}} 
\det(I-q^{n+1-|I|}t\alpha_{\hat{\lambda}}\mid L_0^I)^{(-1)^{n+1+|I|}}.
\end{equation}

For notational convenience, we set $\Gamma = \{0,1,\dots,n\}$.  
\begin{proposition}
{\bf (a)}  The entire function $\det(I-t\alpha_{\hat{\lambda}}\mid L_0^{\Gamma})$ has at most one reciprocal zero of $q$-ordinal equal to $\mu+1$, all other reciprocal zeros have $q$-ordinal $>\mu+1$.  If it has a reciprocal zero of $q$-ordinal equal to $\mu+1$, then all other reciprocal zeros have $q$-ordinal $\geq\mu+2$.  \\
{\bf (b)} The reciprocal zeros of $\det(I-q^{n+1-|I|}t\alpha_{\hat{\lambda}}\mid L_0^I)$ all have $q$-ordinal $\geq\mu+2$ for $I\subsetneq \{0,1,\dots,n\}$.
\end{proposition}

\begin{proof}
Consider first the case $I=\emptyset$, i.e., the entire function $\det(I-q^{n+1}t\alpha_{\hat{\lambda}}\mid L_0^{\emptyset})$.  All reciprocal zeros are divisible by $q^{n+1}$ and $n+1\geq \mu+2$ since $n+1=d(\mu+1)$ and we are assuming $d\geq 2$.  

Now suppose that $I\neq\emptyset$ and let
\[ \omega(I) = \min\{u_{n+1}\mid \text{$x^u\in L^I_0$}\}. \]
Since $x^u\in L_0^I$ if and only if $\sum_{i=0}^n u_i = du_{n+1}$ and $u_i>0$ for $i\in I$, we have
\[ \omega(I) =  \lceil |I|/d\rceil, \]
where $\lceil z\rceil$ denotes the least integer that is $\geq z$.  

It follows from \cite[Proposition~4.2]{AS2} that the first side of the Newton polygon of  $\deg(I-t\alpha_{\hat{\lambda}}\mid L_0^I)$ has slope $\geq \omega(I)$, hence all reciprocal zeros of the entire function $\det(I-q^{n+1-|I|}t\alpha_{\hat{\lambda}}\mid L_0^I)$ have $q$-ordinal greater than or equal to 
\begin{equation}
 n+1-|I| + \lceil|I|/d\rceil.
\end{equation}

First take $I=\Gamma$, i.e., $|I|=n+1$.  In this case the hypothesis that $n+1 = d(\mu+1)$ reduces expression (7.14) to $\mu+1$.  Furthermore, since $(1,\dots,1,\mu+1)$ is the unique element~$u$ with $x^u\in L_0^{\Gamma}$ and $u_{n+1} = \mu+1$, it follows from \cite[Proposition~4.2]{AS2} that the Newton polygon of $\det(I-t\alpha_{\hat{\lambda}}\mid L_0^{\Gamma})$ has a lower bound whose first side has slope $\mu+1$ and length~1.  This implies that $\det(I-t\alpha_{\hat{\lambda}}\mid L_0^{\Gamma})$ has at most one reciprocal zero of $q$-ordinal equal to $\mu+1$ and all other reciprocal zeros have $q$-ordinal $>\mu+1$.  This proves the first sentence of part~(a) of the proposition.  If $\det(I-t\alpha_{\hat{\lambda}}\mid L_0^{\Gamma})$ has a reciprocal zero of $q$-ordinal equal to $\mu+1$, then by \cite[Proposition~4.2]{AS2} the second side of its Newton polygon has slope $\geq\mu+2$.  This proves the second sentence of part (a) of the proposition.

Next take $|I|=n$.  The expression (7.14) reduces to 
\[ 1+\bigg\lceil\frac{n}{d}\bigg\rceil = 1+\bigg\lceil\mu+1-\frac{1}{d}\bigg\rceil = \mu+2 \]
since $d\geq 2$.  Furthermore, expression (7.14) cannot decrease when $|I|$ decreases, which proves part (b) of the proposition.
\end{proof}

Recall from the Introduction that we write 
\[ P_\lambda(t) = {P_\lambda^{(1)}(t)}/{P_\lambda^{(2)}(t)}, \]
where $P^{(1)}_\lambda(t)$ and $P_\lambda^{(2)}(t)$ are relatively prime polynomials with integer coefficients and constant term 1 which satisfy
\[ P_\lambda^{(1)}(q^{-\mu}t),P_\lambda^{(2)}(q^{-\mu}t)\in 1+t{\mathbb Z}[t]. \]
Proposition 7.13, together with Equation (7.12), shows that
\[ P_\lambda^{(2)}(q^{-\mu}t) \equiv 1\pmod{q} \]
and that $P_\lambda^{(1)}(q^{-\mu}t) \pmod{p}$ has degree at most 1 in $t$.
To complete the proof of Proposition 1.7 it suffices, by Proposition 7.13(a), to show that 
\begin{equation}
{\rm Tr}(\alpha_{\hat{\lambda}}\mid L_0^{\Gamma})\equiv q^{\mu+1}\prod_{i=0}^{a-1}\big((-1)^{\mu+1}H(\hat{\lambda}^{p^i})\big)\pmod{pq^{\mu+1}}.
\end{equation}
Using Lemma 5.7, one sees that (7.15) is equivalent to the following assertion.
\begin{proposition} 
For $\lambda\in({\mathbb F}_q^\times)^N$, we have
\[ {\rm Tr}(\alpha_{\hat{\lambda}}\mid L_0^\Gamma) \equiv \prod_{i=0}^{a-1}\theta_{-(p-1){\bf b}}(\hat{\lambda}^{p^i})\pmod{pq^{\mu+1}}. \]
\end{proposition}

\begin{proof}
Consider the series
\[ \theta_0(\hat{\lambda},x) = \sum_{w\in{\mathbb N}A} \theta_{0,w}(\hat{\lambda})x^w. \]
By (7.3) we have
\begin{equation}
\theta_{0,w}(\hat{\lambda}) = \sum_{\substack{u^{(0)},\dots,u^{(a-1)}\in{\mathbb N}A\\ \sum_{i=0}^{a-1} p^iu^{(i)} = w}} \prod_{i=0}^{a-1} \theta_{u^{(i)}}(\hat{\lambda}^{p^i}).
\end{equation}
Let $U\subseteq{\mathbb N}^{n+2}$ be the set of all exponents $u$ such that $x^u\in L_0^\Gamma$.  For $w\in U$, a direct calculation shows that
\begin{equation}
\alpha_{\hat{\lambda}}(x^w) = \sum_{u\in U} \theta_{0,qu-w}(\hat{\lambda})x^u.
\end{equation}
It then follows from the Dwork Trace Formula that
\begin{equation}
{\rm Tr}(\alpha_{\hat{\lambda}}\mid L_0^\Gamma) = \sum_{w\in U} \theta_{0,(q-1)w}(\hat{\lambda}).
\end{equation}

Equation (7.17) gives
\begin{equation}
\theta_{0,(q-1)w}(\hat{\lambda}) = \sum_{\substack{u^{(0)},\dots,u^{(a-1)}\in{\mathbb N}A\\ \sum_{i=0}^{a-1} p^iu^{(i)} = (q-1)w}} \prod_{i=0}^{a-1} \theta_{u^{(i)}}(\hat{\lambda}^{p^i}).
\end{equation}
It follows from (3.21) and (3.23) that
\begin{multline}
{\rm ord}_p\:\theta_{0,(q-1)w}(\hat{\lambda}) \geq \\
\min\bigg\{\sum_{i=0}^{a-1} \frac{u^{(i)}_{n+1}}{p-1}\mid \text{$u^{(0)},\dots,u^{(a-1)}\in{\mathbb N}A$ and $\sum_{i=0}^{a-1} p^iu^{(i)} = (q-1)w$}\bigg\}.
\end{multline}
We prove Proposition 7.16 by studying this estimate for $w\in U$.

Fix $u^{(0)},\dots,u^{(a-1)}\in{\mathbb N}A$ with
\begin{equation}
\sum_{i=0}^{a-1} p^iu^{(i)}= (q-1)w
\end{equation}
and $w\in U$.  We define inductively a sequence $w^{(0)},\dots,w^{(a)}\in U$ such that 
\begin{equation}
u^{(i)} = pw^{(i+1)}-w^{(i)}\quad\text{for $i=0,\dots,a-1$.}
\end{equation}
First of all, take $w^{(0)} = w$.  Eq.~(7.22) shows that $u^{(0)}+w^{(0)}=pw^{(1)}$ for some $w^{(1)}\in{\mathbb Z}^{n+2}$; since $u^{(0)}\in{\mathbb N}A$ and $w^{(0)}\in U$ we conclude that $w^{(1)}\in U$.  Suppose that for some $0< k\leq a-1$ we have defined $w^{(0)},\dots,w^{(k)}\in U$ satisfying (7.23) for $i=0,\dots,k-1$.  Substituting $pw^{(i+1)}-w^{(i)}$ for $u^{(i)}$ for $i=0,\dots,k-1$ in (7.22) gives
\begin{equation}
-w^{(0)}+ p^kw^{(k)} + \sum_{i=k}^{a-1} p^iu^{(i)} = p^aw-w.
\end{equation}
Since $w^{(0)}=w$, we can divide this equation by $p^k$ to get $w^{(k)} + u^{(k)} = pw^{(k+1)}$ for some $w^{(k+1)}\in{\mathbb Z}^{n+2}$.  Since $u^{(k)}\in{\mathbb N}A$ and (by induction) $w^{(k)}\in U$, we conclude that $w^{(k+1)}\in U$.  This completes the inductive construction.  Note that in the special case $k=a-1$, this computation gives $w^{(a)} = w$.

Summing Eq.~(7.23) over $i=0,\dots,a-1$ and using $w^{(0)} = w^{(a)} = w$ gives
\begin{equation}
\sum_{i=0}^{a-1} u^{(i)} = (p-1)\sum_{i=0}^{a-1} w^{(i)},
\end{equation}
hence
\begin{equation}
\sum_{i=0}^{a-1} \frac{u^{(i)}_{n+1}}{p-1} = \sum_{i=0}^{a-1} w^{(i)}_{n+1}.
\end{equation}
Since $w^{(i)}\in U$, we have
\begin{equation}
\begin{cases} w^{(i)}_{n+1} =\mu+1 & \text{if $w^{(i)} = (1,\dots,1,\mu+1)$,} \\ w^{(i)}_{n+1} \geq\mu+2 & \text{if $w^{(i)}\neq (1,\dots,1,\mu+1)$.} \end{cases}
\end{equation}
It now follows from (7.26) that
\begin{equation}
\sum_{i=0}^{a-1}\frac{u^{(i)}_{n+1}}{p-1}\begin{cases} = a(\mu+1) & \text{if $w^{(i)} = (1,\dots,1,\mu+1)$ for $i=0,\dots,a-1$,} \\ \geq a(\mu+1) + 1 & \text{otherwise,} \end{cases}
\end{equation}
so by (7.23), $\sum_{i=0}^{a-1}{u^{(i)}_{n+1}}/{(p-1)} = a(\mu+1)$ if and only if $u^{(i)} = (p-1)(1,\dots,1,\mu+1)$ for all~$i$.

By (7.21), this implies that if $w\neq(1,\dots,1,\mu+1)$, then
\[ \theta_{0,(q-1)w}(\hat{\lambda})\equiv 0\pmod{pq^{\mu+1}}. \]
If $w=(1,\dots,1,\mu+1)$, this implies by (7.20) that
\[ \theta_{0,(q-1)(1,\dots,1,\mu+1)}(\hat{\lambda})\equiv \prod_{i=0}^{a-1}\theta_{(p-1)(1,\dots,1,\mu+1)}(\hat{\lambda}^{p^i}) \pmod{pq^{\mu+1}}. \]
Since $-{\bf b} = (1,\dots,1,\mu+1)$, Equation (7.19) now implies the proposition.
\end{proof}

Let $\lambda\in({\mathbb F}_q^\times)^N$.  In the course of proving Proposition~1.7, we have shown that $\bar{H}(\lambda)\neq 0$ is a necessary and sufficient condition for $\det(I-t\alpha_{\hat{\lambda}}\mid L_0^\Gamma)$ to have a unique reciprocal zero of  $q$-ordinal equal to~$\mu+1$.  To prove the last assertion of Theorem~4.26, it suffices by (7.12) and Proposition 7.13 to prove the following result.
\begin{theorem}
If $\lambda\in({\mathbb F}_q^\times)^N$ and $\bar{H}(\lambda)\neq 0$, then $q^{\mu+1}\prod_{i=0}^{a-1} {\mathcal G}(\hat{\lambda}^{p^i})$ is an eigenvalue of $\alpha_{\hat{\lambda}}$ on $L_0^\Gamma$.
\end{theorem}

Before beginning the proof of Theorem 7.29, we give an alternate description of $\det(I-t\alpha_{\hat{\lambda}}\mid L_0^\Gamma)$.  Let 
\begin{align*}
\hat{M}_- &= \bigg\{ u=(u_0,\dots,u_{n+1})\in({\mathbb Z}_{<0})^{n+2}\mid \sum_{i=0}^n u_i = du_{n+1}\bigg\}, \\
\hat{M}_+ &= \bigg\{ u=(u_0,\dots,u_{n+1})\in({\mathbb Z}_{>0})^{n+2}\mid \sum_{i=0}^n u_i = du_{n+1}\bigg\}.
\end{align*}
Set
\[ B = \bigg\{ \xi^* = \sum_{u\in \hat{M}_-} c_u^*\gamma_0^{pu_{n+1}}x^u\mid \text{$c_u^*\to 0$ as $u\to-\infty$}\bigg\}, \]
a $p$-adic Banach space with norm $\lvert\xi^*\rvert = \sup_{u\in\hat{M}_-}\{\lvert c_u^*\rvert\}$.  We define a pairing $\langle,\rangle :B\times L_0^\Gamma\to{\mathbb C}_p$ as follows.  If $\xi = \sum_{u\in\hat{M}_+} c_u\gamma_0^{pu_{n+1}}x^u\in L_0^\Gamma$ and $\xi^* = \sum_{u\in\hat{M}_-} c_u^*\gamma_0^{pu_{n+1}}x^u\in B$, define
\[ \langle\xi^*,\xi\rangle = \sum_{u\in\hat{M}_+} c_uc^*_{-u}, \]
the constant term of the product $\xi^*\xi$.  This pairing identifies $B$ with the dual space of $L_0^\Gamma$, the space of continuous linear mappings from $L_0^\Gamma$ to ${\mathbb C}_p$ (see Serre\cite[Proposition 3]{S}).  We extend the definition of the mapping $\Phi$ defined in the proof of Proposition~6.9 by setting
\[ \Phi\bigg(\sum_{u\in{\mathbb Z}^n} c_ux^u\bigg) = \sum_{u\in{\mathbb Z}^n} c_ux^{pu}. \]
Consider the formal composition $\alpha^*_{\hat{\lambda}} = \delta_-\circ\text{``multiplication by $\theta_0(\hat{\lambda},x)$''}\circ\Phi^a$.
\begin{proposition}
The operator $\alpha^*_{\hat{\lambda}}$ is an endomorphism of $B$ which is adjoint to $\alpha_{\hat{\lambda}}:L_0^\Gamma\to L_0^\Gamma$.
\end{proposition}

\begin{proof}
Since $\alpha^*_{\hat{\lambda}}$ is the $a$-fold composition of the operators $\delta_-\circ\theta(\hat{\lambda}^{p^i},x))\circ\Phi$ and $\alpha_{\hat{\lambda}}$ is the $a$-fold composition of the operators $\psi\circ\theta(\hat{\lambda}^{p^i},x)$ for $i=0,\dots,a-1$, it suffices to check that $\delta_-\circ\theta(\hat{\lambda},x)\circ\Phi$ is an endomorphism of $B$ adjoint to $\psi\circ\theta(\hat{\lambda},x):
L_0^\Gamma\to L_0^\Gamma$.  Let $\xi^*(x) = \sum_{v\in\hat{M}_-} c_v^*\gamma_0^{pv_{n+1}}x^v\in B$.  The proof that the product $\theta(\hat{\lambda},x)\xi^*(x^p)$ is well-defined is analogous to the proof of convergence of~(5.1).  We have
\[ \delta_-\big(\theta(\hat{\lambda},x)\xi^*(x^p)\big) = \sum_{u\in\hat{M}_-} C_u^*\gamma_0^{pu_{n+1}}x^u, \]
where
\begin{equation}
C_u^* = \sum_{w+pv = u} \theta_w(\hat{\lambda}) c_v^*\gamma_0^{p(v_{n+1}-u_{n+1})}.
\end{equation}
Note that by (3.23)
\begin{equation}
{\rm ord}_p\:\theta_w(\hat{\lambda}) \gamma_0^{p(v_{n+1}-u_{n+1})}\geq \frac{w_{n+1}}{p-1} + \frac{pv_{n+1}}{p-1} -\frac{pu_{n+1}}{p-1} = -u_{n+1} 
\end{equation}
since $w+pv=u$.  Since $c_v^*\to0$ as $v\to-\infty$, this implies that the series on the right-hand side of (7.31) converges.  Furthermore, the estimate (7.32) then shows that $C_u^*\to0$ as $u\to-\infty$.  We conclude that $\delta_-\big(\theta(\hat{\lambda},x)\xi^*(x^p)\big)\in B$.  In fact, (7.32)  implies 
\[ \lvert\delta_-\big(\theta(\hat{\lambda},x)\xi^*(x^p)\big)\rvert\leq \lvert p^{\mu+1}\xi^*(x)\rvert \]
since $u_{n+1}\leq -(\mu+1)$ for all $u\in M_-$.
\end{proof}

\begin{proof}[Proof of Theorem $7.29$]
From Proposition 7.30, it follows by Serre\cite[Proposition~15]{S} that
\begin{equation}
\det(I-t\alpha_{\hat{\lambda}}\mid L_0^\Gamma) = \det(I-t\alpha^*_{\hat{\lambda}}\mid B),
\end{equation}
so to complete the proof of Theorem 7.29 it suffices to show that if $\bar{H}(\lambda)\neq 0$, then $\alpha^*_{\hat{\lambda}}$ has an eigenvector in $B$ with eigenvalue $q^{\mu+1}\prod_{i=0}^{a-1} {\mathcal G}(\hat{\lambda}^{p^i})$.  From Equation~(5.16) we have
\[ \alpha^*\bigg(\frac{G(\Lambda,x)}{G(\Lambda)}\bigg) = p^{\mu+1}{\mathcal G}(\Lambda)\frac{G(\Lambda,x)}{G(\Lambda)}. \]
It follows by iteration that for $m\geq 0$,
\begin{equation}
(\alpha^*)^m\bigg(\frac{G(\Lambda,x)}{G(\Lambda)}\bigg) = p^{m(\mu+1)}\bigg(\prod_{i=0}^{m-1} {\mathcal G}(\Lambda^{p^i})\bigg)\frac{G(\Lambda,x)}{G(\Lambda)}.
\end{equation}
From (4.15) we have
\[ \frac{G(\Lambda,x)}{G(\Lambda)} = \sum_{u\in M_-} \bigg(\gamma_0^{-(p-1)u_{n+1}}\frac{G_u(\Lambda)}{G(\Lambda)}\bigg)\gamma_0^{pu_{n+1}}x^u. \]
Put ${\mathcal G}_u(\Lambda) = G_u(\Lambda)/G(\Lambda)$.  By Proposition 5.15, ${\mathcal G}_u(\Lambda)\in R_u'$, i.e., ${\mathcal G}_u(\Lambda)$ is analytic on ${\mathcal D}_+$.  We may therefore evaluate the ${\mathcal G}_u(\Lambda)$ at $\Lambda = \hat{\lambda}$:
\[ \frac{G(\Lambda,x)}{G(\Lambda)}\bigg|_{\Lambda = \hat{\lambda}} = \sum_{u\in M_-} \big(\gamma_0^{-(p-1)u_{n+1}}{\mathcal G}_u(\hat{\lambda})\big)\gamma_0^{pu_{n+1}}x^u. \]
Since $\gamma_0^{-(p-1)u_{n+1}}\to 0$ as $u\to\infty$, this expression lies in $B$.  It is straightforward to check that the specialization of the left-hand side of Equation (7.34) with $m=a$ at $\Lambda = \hat{\lambda}$ is exactly $\alpha^*_{\hat{\lambda}}\big(G(\Lambda,x)/G(\Lambda)\big|_{\Lambda = \hat{\lambda}}\big)$, so specializing Equation (7.34) with $m=a$ at $\Lambda=\hat{\lambda}$ gives
\begin{multline}
\alpha^*_{\hat{\lambda}}\bigg(\sum_{u\in M_-} \big(\gamma_0^{-(p-1)u_{n+1}}{\mathcal G}_u(\hat{\lambda})\big)\gamma_0^{pu_{n+1}}x^u\bigg) = \\
q^{\mu+1}\bigg(\prod_{i=0}^{a-1} {\mathcal G}(\hat{\lambda}^{p^i})\bigg)\bigg(\sum_{u\in M_-} \big(\gamma_0^{-(p-1)u_{n+1}}{\mathcal G}_u(\hat{\lambda})\big)\gamma_0^{pu_{n+1}}x^u\bigg).
\end{multline}
Equation (7.35) shows that $q^{\mu+1}\prod_{i=0}^{a-1}{\mathcal G}(\hat{\lambda}^{p^i})$ is an eigenvalue of $\alpha^*_{\hat{\lambda}}$.
\end{proof}

\end{document}